\documentclass[11pt,a4paper]{amsart}
\pdfoutput=1
\usepackage[centering,margin=1in]{geometry}
\usepackage{amsmath,amsfonts,amsthm,amssymb,enumitem,graphicx,mathrsfs,mathtools}
\usepackage[stretch=10]{microtype}
\usepackage[unicode,breaklinks=true]{hyperref}
\usepackage{xcolor}
\definecolor{dark-red}{rgb}{0.4,0.15,0.15}
\definecolor{dark-blue}{rgb}{0.15,0.15,0.4}
\definecolor{medium-blue}{rgb}{0,0,0.5}
\hypersetup{
	pdftitle={The Newform \texorpdfstring{$K$}{K}-Type and \texorpdfstring{$p$}{p}-adic Spherical Harmonics},
	pdfauthor={Peter Humphries},
	pdfnewwindow=true,
	colorlinks,
	linkcolor={dark-red},
	citecolor={dark-blue},
	urlcolor={medium-blue}
}
\makeatletter
\newcommand{\bigboxtimes}{
	\mathop{
		\vphantom{\bigotimes} 
		\mathchoice
		{\vcenter{\hbox{\resizebox{\widthof{$\displaystyle\bigotimes$}}{!}{$\boxtimes$}}}}
		{\vcenter{\hbox{\resizebox{\widthof{$\bigotimes$}}{!}{$\boxtimes$}}}}
		{\vcenter{\hbox{\resizebox{\widthof{$\scriptstyle\otimes$}}{!}{$\boxtimes$}}}}
		{\vcenter{\hbox{\resizebox{\widthof{$\scriptscriptstyle\otimes$}}{!}{$\boxtimes$}}}}
	}\displaylimits 
}
\makeatother
\allowdisplaybreaks
\newcommand{\A}{\mathbb{A}}
\newcommand{\C}{\mathbb{C}}
\newcommand{\dee}{\partial}
\newcommand{\e}{\varepsilon}
\newcommand{\F}{\mathbb{F}}
\newcommand{\HH}{\mathcal{H}}
\newcommand{\Mgp}{\mathrm{M}}
\newcommand{\Ngp}{\mathrm{N}}
\newcommand{\OO}{\mathcal{O}}
\newcommand{\Pb}{\mathbb{P}}
\newcommand{\Pgp}{\mathrm{P}}
\newcommand{\pp}{\mathfrak{p}}
\newcommand{\Q}{\mathbb{Q}}
\newcommand{\QQ}{\mathcal{Q}}
\newcommand{\R}{\mathbb{R}}
\newcommand{\WW}{\mathcal{W}}
\newcommand{\Z}{\mathbb{Z}}
\newcommand{\Zgp}{\mathrm{Z}}
\DeclareMathOperator{\ad}{ad}
\DeclareMathOperator{\blockdiag}{blockdiag}
\DeclareMathOperator{\diag}{diag}
\DeclareMathOperator{\End}{End}
\DeclareMathOperator{\GL}{GL}
\DeclareMathOperator{\Hom}{Hom}
\DeclareMathOperator{\Ind}{Ind}
\DeclareMathOperator{\Mat}{Mat}
\DeclareMathOperator{\Ogp}{O}
\DeclareMathOperator{\PGL}{PGL}
\DeclareMathOperator{\sgn}{sgn}
\DeclareMathOperator{\SL}{SL}
\DeclareMathOperator{\Tr}{Tr}
\DeclareMathOperator*{\vol}{vol}
\DeclareMathOperator{\Ugp}{U}
\numberwithin{equation}{section}
\newtheorem{theorem}[equation]{Theorem}
\newtheorem{corollary}[equation]{Corollary}
\newtheorem{lemma}[equation]{Lemma}
\newtheorem{proposition}[equation]{Proposition}
\newtheorem*{question}{Question}
\theoremstyle{remark}
\newtheorem{remark}[equation]{Remark}
\theoremstyle{definition}
\newtheorem{definition}[equation]{Definition}
\begin{document}

\title{The Newform $K$-Type and $p$-adic Spherical Harmonics}

\author{Peter Humphries}

\address{Department of Mathematics, University of Virginia, Charlottesville, VA 22904, USA}

\email{\href{mailto:pclhumphries@gmail.com}{pclhumphries@gmail.com}}

\urladdr{\href{https://sites.google.com/view/peterhumphries/}{https://sites.google.com/view/peterhumphries/}}


\begin{abstract}
Let $K \coloneqq \GL_n(\OO)$ denote the maximal compact subgroup of $\GL_n(F)$, where $F$ is a nonarchimedean local field with ring of integers $\OO$. We study the decomposition of the space of locally constant functions on the unit sphere in $F^n$ into irreducible $K$-modules; for $F = \Q_p$, these are the $p$-adic analogues of spherical harmonics. As an application, we characterise the newform and conductor exponent of a generic irreducible admissible smooth representation of $\GL_n(F)$ in terms of distinguished $K$-types. Finally, we compare our results to analogous results in the archimedean setting.
\end{abstract}

\maketitle

\section{Introduction}

\subsection{The Space \texorpdfstring{$C^{\infty}(S^{n - 1})$}{C\9042\036{}(S\9040\177{}\9040\173{}\80\271)} as a \texorpdfstring{$K$}{K}-Module}

Let $F$ be a nonarchimedean local field with ring of integers $\OO$, maximal prime ideal $\pp$, and uniformiser $\varpi$, so that $\varpi \OO = \pp$ and $\OO/\pp \cong \F_q$ for some prime power $q$, where $\F_q$ is the finite field of order $q$. Thus either $F$ is a finite extension of the $p$-adic numbers $\Q_p$ for some prime $p$ or $F$ is the field of formal Laurent series $\F_q((t))$. We let $K_n \coloneqq \GL_n(\OO)$ denote the maximal compact subgroup of $\GL_n(F)$, which is unique up to conjugacy; when it is clear from context, we write $K$ in place of $K_n$. Throughout we assume that $n \geq 2$.

The group $K_n$ acts transitively on unit sphere $S^{n - 1}$ in $F^n$ via the group action $k \cdot x \coloneqq xk$ for $k \in K_n$ and $x \in S^{n - 1}$, where
\[S^{n - 1} \coloneqq \left\{x = (x_1,\ldots,x_n) \in F^n : \max\{|x_1|,\ldots,|x_n|\} = 1\right\};\]
here $|\cdot|$ denotes the absolute value on $F$ normalised such that $|\varpi| = q^{-1}$. The stabiliser subgroup of $K_n$ with respect to the point $e_n \coloneqq (0,\ldots,0,1) \in S^{n - 1}$ is
\begin{equation}
\label{eqn:Kn-11}
K_{n - 1,1} \coloneqq \left\{ \begin{pmatrix} a & b \\ 0 & 1 \end{pmatrix} \in K_n : a \in K_{n - 1}, \ b \in \Mat_{(n - 1) \times 1}(\OO) \right\},
\end{equation}
which has the structure of the outer semidirect product $K_{n - 1} \ltimes \OO^{n - 1}$. It follows that $S^{n - 1} \cong K_{n - 1,1} \backslash K_n$. Note that $K_{n - 1,1}$ is the maximal compact subgroup of the mirabolic subgroup
\begin{equation}
\label{eqn:mirabolic}
\Pgp_n(F) \coloneqq \left\{\begin{pmatrix} a & b \\ 0 & 1 \end{pmatrix} \in \GL_n(F) : a \in \GL_{n - 1}(F), \ b \in \Mat_{(n - 1) \times 1}(F)\right\}
\end{equation}
of $\GL_n(F)$.

Let $R$ denote the right regular representation of $K_n$ on the space of locally constant functions $C^{\infty}(S^{n - 1})$; thus $(R(k) \cdot f)(x) \coloneqq f(xk)$ for $f \in C^{\infty}(S^{n - 1})$, $k \in K_n$, and $x \in S^{n - 1}$. A natural question to ponder is the following.

\begin{question}
What is the decomposition of the right regular representation $R$ of $K_n$ on $C^{\infty}(S^{n - 1})$ into irreducible smooth representations of $K_n$? Equivalently, which irreducible smooth representations of $K_n$ have a nontrivial $K_{n - 1,1}$-fixed vector?
\end{question}

We study this problem in \hyperref[sect:spherical]{Section \ref*{sect:spherical}}. While the general classification of irreducible smooth representations of $K_n$ is unknown for $n \geq 3$, we show that the representations having a nontrivial $K_{n - 1,1}$-fixed vector can be explicitly described; in particular, a precise resolution of this question is given in \hyperref[thm:Hchimirred]{Theorem \ref*{thm:Hchimirred}}. The irreducible smooth representations of interest are indexed by pairs $(\chi,m)$ of characters $\chi$ of $\OO^{\times}$ and nonnegative integers $m \geq c(\chi)$, where $c(\chi)$ denotes the conductor exponent of $\chi$. The character $\chi$ is the central character of this representation, while $m$ is its level, namely the minimal nonnegative integer for which this representation factors through $\GL_n(\OO / \pp^m)$. We denote by $(\tau_{\chi,m},\HH_{\chi,m}(S^{n - 1}))$ the irreducible smooth representation associated to such a pair, where $\HH_{\chi,m}(S^{n - 1})$ is a finite-dimensional vector subspace of $C^{\infty}(S^{n - 1})$; this may be thought of as the nonarchimedean analogue of the vector space of spherical harmonics of degree $m$ on the unit sphere in $\R^n$. In particular, when $F = \Q_p$, we may think of $\HH_{\chi,m}(S^{n - 1})$ as being the space of $p$-adic spherical harmonics of character $\chi$ and level $m$.

\begin{remark}
The decomposition attained in \hyperref[thm:Hchimirred]{Theorem \ref*{thm:Hchimirred}} is not, in essence, fundamentally new. With $K_n = \GL_n(\OO)$ replaced by $\SL_n(\OO)$, this decomposition was previously achieved (with scant proofs) by Petrov \cite{Pet82} in a seemingly neglected paper; our method follows that sketched by Petrov and achieves the same decomposition when $n \geq 3$. A closely related result, via slightly different methods, is due to Hill \cite[Proposition 3.3]{Hil94}, who studies instead the decomposition of $C^{\infty}(\Zgp(\OO) K_{n - 1,1} \backslash K_n)$, where $\Zgp(\OO)$ denotes the centre of $K$; note that we may identify $\Zgp(\OO) K_{n - 1,1} \backslash K_n$ with $(n - 1)$-dimensional $F$-projective space $F\Pb^{n - 1}$.
\end{remark}

\subsection{Zonal Spherical Functions}

We next study the subspace of locally constant functions in $C^{\infty}(S^{n - 1})$ that are $K_{n - 1,1}$-invariant in \hyperref[sect:zonal]{Section \ref*{sect:zonal}}. Each irreducible subspace $\HH_{\chi,m}(S^{n - 1})$ of $C^{\infty}(S^{n - 1})$ has a one-dimensional subspace of $K_{n - 1,1}$-invariant functions, so that there exists a unique function $P_{\chi,m}^{\circ} \in \HH_{\chi,m}(S^{n - 1})$ satisfying $P_{\chi,m}^{\circ}(xk') = P_{\chi,m}^{\circ}(x)$ for all $x \in S^{n - 1}$ and $k' \in K_{n - 1,1}$ and $P_{\chi,m}^{\circ}(e_n) = 1$. We call $P_{\chi,m}^{\circ}$ the zonal spherical function on $S^{n - 1}$ of character $\chi$ and level $m$. These are the nonarchimedean analogues of zonal spherical harmonics (or ultraspherical polynomials).

We give an explicit formula for $P_{\chi,m}^{\circ}$ in \hyperref[prop:Pcirc]{Proposition \ref*{prop:Pcirc}}. Using this, we prove a nonarchimedean analogue of the addition formula for spherical harmonics in \hyperref[lem:addthm]{Lemma \ref*{lem:addthm}}. A consequence of this is \hyperref[cor:reproducing]{Corollary \ref*{cor:reproducing}}, which states that $(\dim \HH_{\chi,m}(S^{n - 1})) P_{\chi,m}^{\circ}$ is the reproducing kernel for $\HH_{\chi,m}(S^{n - 1})$; this implies that each element of $\HH_{\chi,m}(S^{n - 1})$ is \emph{equal} to a matrix coefficient of $\tau_{\chi,m}$.

\subsection{The Newform and the Conductor Exponent}

We apply this theory of $p$-adic spherical harmonics in \hyperref[sect:Ktype]{Section \ref*{sect:Ktype}} to previous work of Jacquet, Piatetski-Shapiro, and Shalika \cite{JP-SS81}. They prove the existence of a distinguished vector, the newform, associated to a given generic irreducible admissible representation $(\pi,V_{\pi})$ of $\GL_n(F)$ (or more generally to an induced representation of Langlands type). This vector is invariant under a certain congruence subgroup $K_1(\pp^m)$ of $K_n$ and is the minimal such nontrivial vector in the sense that there are no nontrivial vectors invariant under $K_1(\pp^{\ell})$ with $\ell < m$. This minimal value of $m$ is called the conductor exponent of $\pi$ and is denoted by $c(\pi)$.

We give alternative characterisations of the newform and conductor exponent in \hyperref[thm:newform]{Theorem \ref*{thm:newform}}. Using the work of Jacquet, Piatetski-Shapiro, and Shalika \cite{JP-SS81}, we show that the conductor exponent $c(\pi)$ of $\pi$ is the minimal nonnegative integer $m$ for which there exists a nontrivial $K_{n - 1,1}$-invariant $\tau_{\chi,m}$-isotypic vector in $V_{\pi}$ for some character $\chi$ of $\OO^{\times}$; necessarily, $\chi$ must then be equal to $\chi_{\pi}$, the restriction from $F^{\times}$ to $\OO^{\times}$ of the central character $\omega_{\pi}$ of $\pi$. We also show that the newform is precisely the nonzero vector, unique up to scalar multiplication, that is $K_{n - 1,1}$-invariant and $\tau_{\chi_{\pi},c(\pi)}$-isotypic; for this reason, we name $\tau_{\chi_{\pi},c(\pi)}$ the newform $K$-type. Our methods also allow us to describe precisely in \hyperref[prop:newform]{Proposition \ref*{prop:newform}} the multiplicity with which a representation $\tau_{\chi,m}$ occurs in the $K$-type decomposition of $\pi$.

These different characterisations open up new avenues of approach to studying properties of the newform. We give a simple example of one such property in \hyperref[cor:matrixcoeff]{Corollary \ref*{cor:matrixcoeff}}, where we show that the matrix coefficient associated to the newform may be explicitly described in terms of the zonal spherical function $P_{\chi_{\pi},c(\pi)}^{\circ} \in \HH_{\chi_{\pi},c(\pi)}(S^{n - 1})$.

\subsection{The Archimedean Theory}

Finally, in \hyperref[sect:archimedean]{Section \ref*{sect:archimedean}}, we compare our results to analogous results in the archimedean setting. The archimedean analogue of the maximal compact subgroup $K_n = \GL_n(\OO)$ of $\GL_n(F)$ is the orthogonal group $\Ogp(n)$ if $F = \R$ and the unitary group $\Ugp(n)$ if $F = \C$. The decomposition of $C^{\infty}(S^{n - 1})$ into irreducible $K_n$-modules in these settings is well-known: it is the classical theory of spherical harmonics. These are the restriction to the unit sphere of homogeneous harmonic polynomials of a given degree if $F = \R$ or of a given bidegree if $F = \C$. In both cases, there exist analogues of the the zonal spherical functions $P_{\chi,m}^{\circ}$.

The archimedean analogue of the theory of newforms and conductor exponents of a generic irreducible admissible smooth representation of $\GL_n(F)$ was recently developed by the author \cite{Hum20}. Over archimedean local fields, the approach of Jacquet, Piatetski-Shapiro, and Shalika \cite{JP-SS81} via congruence subgroups is no longer applicable, and instead the development of the theory of newforms is via distinguished $K_n$-types. Thus our alternative characterisation in the nonarchimedean setting proven in \hyperref[thm:newform]{Theorem \ref*{thm:newform}} serves to unify these two different settings.

\section{\texorpdfstring{$p$}{p}-adic Spherical Harmonics}
\label{sect:spherical}

Our goal is to decompose the right regular representation of $K$ on $C^{\infty}(S^{n - 1})$ into irreducible smooth representations. Here $C^{\infty}(S^{n - 1})$ is the space of locally constant functions $f : S^{n - 1} \to \C$, namely functions that are right-invariant by a compact open subgroup of $K$. We define a $K$-invariant inner product on $C^{\infty}(S^{n - 1}) \ni f_1,f_2$ via
\begin{equation}
\label{eqn:innerproduct}
\langle f_1,f_2\rangle \coloneqq \int_{K} f_1(e_n k) \overline{f_2}(e_n k) \, dk,
\end{equation}
where $dk$ denotes the Haar probability measure on the compact group $K$. We also recall that given a closed subgroup $H$ of a compact profinite group $G$ and a one-dimensional representation $\psi$ of $H$, the induced representation $\Ind_H^G \psi$ is the vector space of locally constant functions $\phi : G \to \C$ that satisfy $\phi(hg) = \psi(h) \phi(g)$ for all $g \in G$ and $h \in H$, upon which $G$ acts via right translations.

Our first observation is that the space $C^{\infty}(S^{n - 1})$ may be identified with the induced representation of the trivial representation of $K_{n - 1,1}$.

\begin{lemma}
\label{lem:IndKn-1,1}
As $K$-modules, the space $C^{\infty}(S^{n - 1})$ is isomorphic to $\Ind_{K_{n - 1,1}}^{K} 1$.
\end{lemma}

\begin{proof}
Given $f \in C^{\infty}(S^{n - 1})$, define $\phi(k) \coloneqq f(e_n k)$; then clearly $\phi \in \Ind_{K_{n - 1,1}}^{K} 1$. Conversely, given $\phi \in \Ind_{K_{n - 1,1}}^{K} 1$, $\phi(k)$ is dependent only on $e_n k$. Since for each $x \in S^{n - 1}$, there exists some $k \in K$ for which $x = e_n k$, the function $f(x) \coloneqq \phi(k)$ is well-defined, which gives us an element of $C^{\infty}(S^{n - 1})$.
\end{proof}

\subsection{Reduction to \texorpdfstring{$K_1(\pp^m)$}{K\9040\201(pm)}}

We next study subspaces of $C^{\infty}(S^{n - 1})$ consisting of locally constant functions invariant under certain congruence subgroups. For each nonnegative integer $m$, let $K(\pp^m)$ denote the principal congruence subgroup of level $m$ of $K$, namely
\[K(\pp^m) \coloneqq \left\{k \in K : k - 1_n \in \Mat_{n \times n}(\pp^m)\right\},\]
where $1_n$ denotes the $n \times n$ identity matrix. This is a normal subgroup of $K$. These subgroups allow us to construct a filtration of $C^{\infty}(S^{n - 1})$ via the subspaces
\[C^{\infty}(S^{n - 1})^{K(\pp^m)} \coloneqq \left\{f \in C^{\infty}(S^{n - 1}) : f(xk) = f(x) \text{ for all $x \in S^{n - 1}$ and $k \in K(\pp^m)$}\right\}\]
of $K(\pp^m)$-invariant locally constant functions on $S^{n - 1}$. We observe that $C^{\infty}(S^{n - 1})^{K(\pp^m)}$ is contained in $C^{\infty}(S^{n - 1})^{K(\pp^{\ell})}$ for all nonnegative integers $m \leq \ell$. Furthermore, the union $\bigcup_{m = 0}^{\infty} C^{\infty}(S^{n - 1})^{K(\pp^m)}$ is equal to $C^{\infty}(S^{n - 1})$.

We also define the congruence subgroup $K_1(\pp^m)$ of $K$ by
\[K_1(\pp^m) \coloneqq \left\{ \begin{pmatrix} a & b \\ c & d \end{pmatrix} \in K : c \in \Mat_{1 \times (n - 1)}(\pp^m), \ d - 1 \in \pp^m \right\}\]
for each nonnegative integer $m$, so that $K_1(\pp^0) = K$, while if $m \geq 1$, $\begin{psmallmatrix} a & b \\ c & d \end{psmallmatrix} \in K_1(\pp^m)$ implies that $a \in K_{n - 1}$ and $d \in \OO^{\times}$; we make note of the fact that $K_1(\pp^m)$ contains $K_1(\pp^{\ell})$ and $K(\pp^{\ell})$ as subgroups whenever $\ell \geq m$.

\begin{lemma}
\label{lem:IndK1}
As $K$-modules, $C^{\infty}(S^{n - 1})^{K(\pp^m)}$ is isomorphic to $\Ind_{K_1(\pp^m)}^{K} 1$.
\end{lemma}

\begin{proof}
From \hyperref[lem:IndKn-1,1]{Lemma \ref*{lem:IndKn-1,1}}, $C^{\infty}(S^{n - 1})^{K(\pp^m)}$ is isomorphic as a $K$-module to the space of locally constant functions $\phi : K \to \C$ that satisfy $\phi(k' k k'') = \phi(k)$ for all $k' \in K_{n - 1,1}$, $k \in K$, and $k'' \in K(\pp^m)$. Since $K(\pp^m)$ is a normal subgroup of $K$, this is equal to the space of locally constant functions $\phi : K \to \C$ that satisfy $\phi(k' k'' k) = \phi(k)$ for all $k' \in K_{n - 1,1}$, $k'' \in K(\pp^m)$, and $k \in K$. The result then follows from the fact that $K_{n - 1,1} K(\pp^m) = K_1(\pp^m)$, which is clearly true if $m = 0$; for $m \geq 1$, it is immediate that $K_{n - 1,1} K(\pp^m) \subseteq K_1(\pp^m)$, while the fact that $K_{n - 1,1} K(\pp^m) \supseteq K_1(\pp^m)$ can be seen directly, since for $\begin{psmallmatrix} a & b \\ c & d \end{psmallmatrix} \in K_1(\pp^m)$, we have that
\[\begin{pmatrix} a & b \\ c & d \end{pmatrix} = \begin{pmatrix} a & bd^{-1} \\ 0 & 1 \end{pmatrix} \begin{pmatrix} 1_{n - 1} - a^{-1} bd^{-1}c & 0 \\ c & d \end{pmatrix}.\qedhere\]
\end{proof}

\begin{corollary}
We have that
\begin{equation}
\label{eqn:Scmdim}
\dim C^{\infty}(S^{n - 1})^{K(\pp^m)} = \begin{dcases*}
1 & if $m = 0$,	\\
q^{(m - 1)n} (q^n - 1) & if $m \geq 1$.
\end{dcases*}
\end{equation}
\end{corollary}

\begin{proof}
We have that $\dim C^{\infty}(S^{n - 1})^{K(\pp^m)} = \dim \Ind_{K_1(\pp^m)}^{K} 1$ from \hyperref[lem:IndK1]{Lemma \ref*{lem:IndK1}}. Since the trivial representation is one-dimensional, the dimension of the monomial representation $\Ind_{K_1(\pp^m)}^{K} 1$ is simply the index $[K : K_1(\pp^m)]$ of the subgroup $K_1(\pp^m)$ in $K$. We claim that this is precisely the right-hand side of \eqref{eqn:Scmdim}. Indeed, we have that
\[\left[K : K_1(\pp^m)\right] = \left[K/K(\pp^m) : K_1(\pp^m)/K(\pp^m)\right] = \left[\GL_n(\OO/\pp^m) : \Pgp_n(\OO/\pp^m)\right],\]
where $\Pgp_n$ denotes the mirabolic subgroup \eqref{eqn:mirabolic}. We may identify $\Pgp_n(\OO/\pp^m) \backslash \GL_n(\OO/\pp^m)$ with
\begin{equation}
\label{eqn:barSn-1}
\overline{S}^{n - 1} \coloneqq \left\{\overline{x} = (\overline{x_1},\ldots,\overline{x_n}) \in (\OO/\pp^m)^n : \max\{|\overline{x_1}|,\ldots,|\overline{x_n}|\} = 1\right\},
\end{equation}
where $\overline{x_j}$ denotes a coset of $\pp^m$ in $\OO$ of the form $x_j + \pp^m$ with a representative $x_j \in \OO$, and $|\overline{x_j}|$ is equal to $|x_j|$ if the representative $x_j$ is not in $\pp^m$ and is equal to $q^{-m}$ otherwise; it is not hard to see that $|\overline{x_j}|$ is independent of the choice of representative $x_j$ and hence is well-defined. Now a straightforward induction on $m$ argument shows that the set $\overline{S}^{n - 1}$ has cardinality $q^{(m - 1)n} (q^n - 1)$.
\end{proof}

\subsection{Reduction to \texorpdfstring{$K_0(\pp^m)$}{K\9040\200(pm)}}

Next, we decompose $C^{\infty}(S^{n - 1})^{K(\pp^m)}$ further into subspaces of locally constant functions with a prescribed central character, noting that the centre of $K$ is $\Zgp(\OO) \coloneqq \{z(a) \coloneqq \diag(a,\ldots,a) \in K : a \in \OO^{\times}\}$. Given a character $\chi$ of $\OO^{\times}$, we define the subspace of $C^{\infty}(S^{n - 1})$ consisting of $K(\pp^m)$-invariant (and hence locally constant) functions on $S^{n - 1}$ with central character $\chi$ by
\[C^{\infty}(S^{n - 1})_{\chi}^{K(\pp^m)} \coloneqq \left\{f \in C^{\infty}(S^{n - 1})^{K(\pp^m)} : f(ax) = \chi(a) f(x) \text{ for all $a \in \OO^{\times}$ and $x \in S^{n - 1}$}\right\}.\]
Clearly $C^{\infty}(S^{n - 1})_{\chi}^{K(\pp^m)}$ is trivial if $m < c(\chi)$, where $c(\chi)$ denotes the conductor exponent of $\chi$, namely the least nonnegative integer $m$ for which $\chi$ is trivial on $1 + \pp^m$. Moreover, $C^{\infty}(S^{n - 1})_{\chi}^{K(\pp^m)}$ is contained in $C^{\infty}(S^{n - 1})_{\chi}^{K(\pp^{\ell})}$ for all $\ell \geq m \geq c(\chi)$ and any two subspaces $C^{\infty}(S^{n - 1})_{\chi_1}^{K(\pp^{m_1})}$ and $C^{\infty}(S^{n - 1})_{\chi_2}^{K(\pp^{m_2})}$ are mutually orthogonal with respect to the inner product \eqref{eqn:innerproduct} whenever $\chi_1 \neq \chi_2$.

For each nonnegative integer $m$, the congruence subgroup $K_0(\pp^m)$ of $K$ is defined by
\[K_0(\pp^m) \coloneqq \left\{ \begin{pmatrix} a & b \\ c & d \end{pmatrix} \in K : c \in \Mat_{1 \times (n - 1)}(\pp^m)\right\},\]
so that $K_0(\pp^0) = K$, while if $m \geq 1$, $\begin{psmallmatrix} a & b \\ c & d \end{psmallmatrix} \in K_0(\pp^m)$ implies that $a \in K_{n - 1}$ and $d \in \OO^{\times}$. We observe that $K_0(\pp^m)$ contains $K_0(\pp^{\ell})$, $K_1(\pp^{\ell})$, and $K(\pp^{\ell})$ as subgroups whenever $\ell \geq m$.

Let $\widehat{\OO^{\times}}$ denote the set of (continuous) characters $\chi : \OO^{\times} \to \C^{\times}$; necessarily the image of such a character is in the unit circle $\{z \in \C^{\times} : |z| = 1\}$. For $\chi \in \widehat{\OO^{\times}}$, let $\psi_{\chi}$ be the character of $K_0(\pp^{c(\chi)}) \ni k_0 = \begin{psmallmatrix} a & b \\ c & d \end{psmallmatrix}$ given by $\psi_{\chi}(k_0) \coloneqq \chi(d)$, which is a one-dimensional representation of $K_0(\pp^{c(\chi)})$; by restriction, this is also a one-dimensional representation of $K_0(\pp^m)$ whenever $m \geq c(\chi)$.

\begin{lemma}
\label{lem:IndK0}
For $\chi \in \widehat{\OO^{\times}}$ and $m \geq c(\chi)$, $C^{\infty}(S^{n - 1})_{\chi}^{K(\pp^m)}$ is isomorphic as a $K$-module to $\Ind_{K_0(\pp^m)}^{K} \psi_{\chi}$.
\end{lemma}

\begin{proof}
From \hyperref[lem:IndK1]{Lemma \ref*{lem:IndK1}}, $C^{\infty}(S^{n - 1})_{\chi}^{K(\pp^m)}$ is isomorphic as a $K$-module to the space of locally constant functions $\phi : K \to \C$ that satisfy $\phi(z(a) k_1 k) = \chi(a) \phi(k)$ for all $z(a) \in \Zgp(\OO)$, $k_1 \in K_1(\pp^m)$, and $k \in K$. It remains to note that $\Zgp(\OO) K_1(\pp^m) = K_0(\pp^m)$.
\end{proof}

\begin{corollary}
We have that
\begin{equation}
\label{eqn:Scchimdim}
\dim C^{\infty}(S^{n - 1})_{\chi}^{K(\pp^m)} = \begin{dcases*}
1 & if $m = c(\chi) = 0$,	\\
q^{(m - 1)(n - 1)} \frac{q^n - 1}{q - 1} & if $m \geq \max\{c(\chi),1\}$,	\\
0 & otherwise.
\end{dcases*}
\end{equation}
\end{corollary}

\begin{proof}
We have that $\dim C^{\infty}(S^{n - 1})_{\chi}^{K(\pp^m)} = \dim \Ind_{K_0(\pp^m)}^{K} \psi_{\chi}$ from \hyperref[lem:IndK0]{Lemma \ref*{lem:IndK0}}. Since $\psi_{\chi}$ is one-dimensional, the dimension of the monomial representation $\Ind_{K_0(\pp^m)}^{K} \psi_{\chi}$ is simply the index of $K_0(\pp^m)$ in $K$, so that
\[\dim \Ind_{K_0(\pp^m)}^{K} \psi_{\chi} = \begin{dcases*}
\left[K : K_0(\pp^m)\right] & if $m \geq c(\chi)$,	\\
0 & otherwise.
\end{dcases*}\]
This precisely the right-hand side of \eqref{eqn:Scchimdim}, since $K_1(\pp^m)$ is a normal subgroup of $K_0(\pp^m)$ with quotient isomorphic to the finite abelian group $\OO^{\times} / (1 + \pp^m)$, which has cardinality $q^{m - 1} (q - 1)$, together with the fact that $[K:K_1(\pp^m)] = [K:K_0(\pp^m)] [K_0(\pp^m) : K_1(\pp^m)]$.
\end{proof}

\begin{corollary}
We have the orthogonal decomposition
\begin{equation}
\label{eqn:CKmdecomp}
C^{\infty}(S^{n - 1})^{K(\pp^m)} = \bigoplus_{\substack{\chi \in \widehat{\OO^{\times}} \\ 0 \leq c(\chi) \leq m}} C^{\infty}(S^{n - 1})_{\chi}^{K(\pp^m)}.
\end{equation}
\end{corollary}

\begin{proof}
Since $K_1(\pp^m)$ is a normal subgroup of $K_0(\pp^m)$ with quotient isomorphic to $\OO^{\times} / (1 + \pp^m)$, we have that
\begin{equation}
\label{eqn:K0toK1ind}
\Ind_{K_1(\pp^m)}^{K_0(\pp^m)} 1 = \bigoplus_{\substack{\chi \in \widehat{\OO^{\times}} \\ 0 \leq c(\chi) \leq m}} \psi_{\chi}
\end{equation}
and so by inducing in stages,
\[\Ind_{K_1(\pp^m)}^{K} 1 = \bigoplus_{\substack{\chi \in \widehat{\OO^{\times}} \\ 0 \leq c(\chi) \leq m}} \Ind_{K_0(\pp^m)}^{K} \psi_{\chi}.\]
Together with \hyperref[lem:IndK0]{Lemma \ref*{lem:IndK0}}, this gives the orthogonal decomposition \eqref{eqn:CKmdecomp}.
\end{proof}

\subsection{Irreducible Representations}

The spaces $C^{\infty}(S^{n - 1})_{\chi}^{K(\pp^m)}$ are not irreducible if $m > c(\chi)$ since $C^{\infty}(S^{n - 1})_{\chi}^{K(\pp^m)}$ contains $C^{\infty}(S^{n - 1})_{\chi}^{K(\pp^{\ell})}$ for all $\ell \in \{c(\chi), \ldots, m - 1\}$. Thus we are led to study the orthogonal complement $C^{\infty}(S^{n - 1})_{\chi}^{K(\pp^m)} \ominus C^{\infty}(S^{n - 1})_{\chi}^{K(\pp^{m - 1})}$ of $C^{\infty}(S^{n - 1})_{\chi}^{K(\pp^{m - 1})}$ in $C^{\infty}(S^{n - 1})_{\chi}^{K(\pp^m)}$ with respect to the inner product \eqref{eqn:innerproduct}. For $m \geq c(\chi)$, define
\begin{equation}
\label{eqn:H}
\HH_{\chi,m}(S^{n - 1}) \coloneqq \begin{dcases*}
C^{\infty}(S^{n - 1})_{\chi}^{K(\pp^{c(\chi)})} & if $m = c(\chi)$,	\\
C^{\infty}(S^{n - 1})_{\chi}^{K(\pp^m)} \ominus C^{\infty}(S^{n - 1})_{\chi}^{K(\pp^{m - 1})} & if $m > c(\chi)$,
\end{dcases*}
\end{equation}
As $K$-modules,
\begin{equation}
\label{eqn:HtoInd}
\HH_{\chi,m}(S^{n - 1}) \cong \begin{dcases*}
\Ind_{K_0(\pp^{c(\chi)})}^{K} \psi_{\chi} & if $m = c(\chi)$,	\\
\Ind_{K_0(\pp^m)}^{K} \psi_{\chi} \ominus \Ind_{K_0(\pp^{m - 1})}^{K} \psi_{\chi} & if $m > c(\chi)$.
\end{dcases*}
\end{equation}

\begin{lemma}
\label{lem:dimH}
We have that
\[\dim \HH_{\chi,m}(S^{n - 1}) = \begin{dcases*}
1 & if $c(\chi) = m = 0$,	\\
q \frac{q^{n - 1} - 1}{q - 1} & if $c(\chi) = 0$ and $m = 1$,	\\
q^{(c(\chi) - 1)(n - 1)} \frac{q^n - 1}{q - 1} & if $c(\chi) = m \geq 1$,	\\
q^{(m - 2)(n - 1)} \frac{(q^n - 1)(q^{n - 1} - 1)}{q - 1} & if $m > \max\{c(\chi),1\}$.
\end{dcases*}\]
\end{lemma}

\begin{proof}
This follows immediately from \eqref{eqn:Scchimdim} and \eqref{eqn:H}.
\end{proof}

The spaces $\HH_{\chi,m}(S^{n - 1})$ are $K$-invariant; furthermore, any two subspaces $\HH_{\chi_1,m_1}(S^{n - 1})$ and $\HH_{\chi_1,m_1}(S^{n - 1})$ are mutually orthogonal whenever either $\chi_1 \neq \chi_2$ or $m_1 \neq m_2$. We claim that these subspaces are irreducible, which thereby completes the decomposition of $C^{\infty}(S^{n - 1})$ into irreducible $K$-modules.

\begin{theorem}
\label{thm:Hchimirred}
For each character $\chi \in \widehat{\OO^{\times}}$ and for each integer $m \geq c(\chi)$, the $K$-module $\HH_{\chi,m}(S^{n - 1})$ is irreducible, and we have the orthogonal decompositions
\begin{align}
\label{eqn:CchiKmdecomp}
C^{\infty}(S^{n - 1})_{\chi}^{K(\pp^m)} & = \bigoplus_{\ell = c(\chi)}^{m} \HH_{\chi,\ell}(S^{n - 1}),	\\
\label{eqn:Cdecomp}
C^{\infty}(S^{n - 1}) & = \bigoplus_{m = 0}^{\infty} \bigoplus_{\substack{\chi \in \widehat{\OO^{\times}} \\ 0 \leq c(\chi) \leq m}} \HH_{\chi,m}(S^{n - 1}).
\end{align}
\end{theorem}

For $n = 2$, Casselman \cite[Proposition 1]{Cas73b} studies the decomposition into irreducible representations of $\Ind_{K_0(\pp^m)}^{K} \psi_{\chi} \cong C^{\infty}(S^{n - 1})_{\chi}^{K(\pp^m)}$; see also \cite[Section 3.3]{BP17} for the case $\chi = 1$.

Let $\widehat{K}$ denote the set of equivalence classes of irreducible smooth representations of $K$, and write $\tau_{\chi,m}$ for the representation in $\widehat{K}$ given by right translations on the finite-dimensional vector space $\HH_{\chi,m}(S^{n - 1})$. We have now classified precisely which representations in $\widehat{K}$ have a $K_{n - 1,1}$-fixed vector.

\begin{corollary}
\label{cor:onedim}
For every irreducible smooth representation $\tau \in \widehat{K}$,
\[\dim \Hom_{K_{n - 1,1}}\left(1,\tau|_{K_{n - 1,1}}\right) = \begin{dcases*}
1 & if $\tau = \tau_{\chi,m}$ for some $\chi \in \widehat{\OO^{\times}}$ and $m \geq c(\chi)$,	\\
0 & otherwise.
\end{dcases*}\]
In particular, the subspace
\[\HH_{\chi,m}(S^{n - 1})^{K_{n - 1,1}} \coloneqq \left\{f \in \HH_{\chi,m}(S^{n - 1}) : f(xk') = f(x) \text{ for all $x \in S^{n - 1}$ and $k' \in K_{n - 1,1}$}\right\}\]
of $K_{n - 1,1}$-invariant functions in $\HH_{\chi,m}(S^{n - 1})$ is one-dimensional.
\end{corollary}

\begin{proof}
We observe that $\tau_{\chi_1,m_1}$ is isomorphic to $\tau_{\chi_2,m_2}$ if and only if $\chi_1 = \chi_2$ and $m_1 = m_2$ by examining the dimensions and central characters of these representations. It follows that $C^{\infty}(S^{n - 1})$ is multiplicity-free, so that $(K_n, K_{n - 1,1})$ is a Gelfand pair. The result then follows via \eqref{eqn:Cdecomp} and Frobenius reciprocity.
\end{proof}

The proof of \hyperref[thm:Hchimirred]{Theorem \ref*{thm:Hchimirred}} requires the following lemma.

\begin{lemma}
\label{lem:doublecoset}
For each nonnegative integer $m$, we have the double coset decomposition
\[K = \bigsqcup_{\ell = 0}^{m} K_0(\pp^m) \begin{pmatrix} 1_{n - 1} & 0 \\ \varpi^{\ell} e_{n - 1} & 1 \end{pmatrix} K_0(\pp^m).\]
\end{lemma}

For $n = 2$, this follows from \cite[Lemma 1]{Cas73b} and \cite[Lemma 2.1.1]{Sch02}, while the same result with $K_n = \GL_n(\OO)$ replaced by $\SL_n(\OO)$ is implicit in the work of Chang \cite{Cha98}.

\begin{proof}
Let $\begin{psmallmatrix} a & b \\ c & d \end{psmallmatrix} \in K$ with $a \in \Mat_{(n - 1) \times (n - 1)}(\OO)$, $b \in \Mat_{(n - 1) \times 1}(\OO)$, $c \in \Mat_{1 \times (n - 1)}(\OO)$, and $d \in \OO$. There are three cases to consider.
\begin{enumerate}[leftmargin=*]
\item[(1)] If $\max\{|c_1|,\ldots,|c_{n - 1}|\} \leq q^{-m}$, then $\begin{psmallmatrix} a & b \\ c & d \end{psmallmatrix} \in K_0(\pp^m) = K_0(\pp^m) \begin{psmallmatrix} 1_{n - 1} & 0 \\ \varpi^{m} e_{n - 1} & 1 \end{psmallmatrix} K_0(\pp^m)$.
\item[(2)] If $\max\{|c_1|,\ldots,|c_{n - 1}|\} = q^{-\ell}$ for some $\ell \in \{1,\ldots,m - 1\}$, then $a \in K_{n - 1}$; as $\varpi^{-\ell} ca^{-1} \in S^{n - 2}$, there exists some $\alpha \in K_{n - 1}$ such that $e_{n - 1} \alpha^{-1} = \varpi^{-\ell} ca^{-1}$, and we have that
\[\begin{pmatrix} a & b \\ c & d \end{pmatrix} = \begin{pmatrix} \alpha & 0 \\ 0 & 1 \end{pmatrix} \begin{pmatrix} 1_{n - 1} & 0 \\ \varpi^{\ell} e_{n - 1} & 1 \end{pmatrix} \begin{pmatrix} \alpha^{-1} a & \alpha^{-1} b \\ 0 & -ca^{-1} b + d \end{pmatrix}.\]
\item[(3)] Finally, if $\max\{|c_1|,\ldots,|c_{n - 1}|\} = 1$, then $c \in S^{n - 2}$. By \cite[Lemma 3]{Cha98}, there exists $\beta \in \Mat_{(n - 1) \times 1}(\OO)$ such that $\det(a - \beta c) \in \OO^{\times}$, so that $a - \beta c \in K_{n - 1}$. As $c(a - \beta c)^{-1} \in S^{n - 2}$, there exists some $\alpha \in K_{n - 1}$ such that $e_{n - 1} \alpha^{-1} = c(a - \beta c)^{-1}$, and we have that
\[\begin{pmatrix} a & b \\ c & d \end{pmatrix} = \begin{pmatrix} \alpha & \beta \\ 0 & 1 \end{pmatrix} \begin{pmatrix} 1_{n - 1} & 0 \\ e_{n - 1} & 1 \end{pmatrix} \begin{pmatrix} \alpha^{-1} (a - \beta c) & \alpha^{-1} (b - \beta d) \\ 0 & -c(a - \beta c)^{-1} (b - \beta d) + d \end{pmatrix}.\qedhere\]
\end{enumerate}
\end{proof}

\begin{proof}[Proof of {\hyperref[thm:Hchimirred]{Theorem \ref*{thm:Hchimirred}}}]
For $m \geq c(\chi)$, we identify $\End_{K}(\Ind_{K_0(\pp^m)}^{K} \psi_{\chi})$ with the space of locally constant functions $\phi : K \to \C$ that satisfy $\phi(k_0 k k_0') = \psi_{\chi}(k_0) \psi_{\chi}(k_0') \phi(k)$ for all $k \in K$ and $k_0,k_0' \in K_0(\pp^m)$. From \hyperref[lem:doublecoset]{Lemma \ref*{lem:doublecoset}}, we deduce that for each integer $m \geq c(\chi)$,
\[\dim \End_{K}\left(\Ind_{K_0(\pp^m)}^{K} \psi_{\chi}\right) = m - c(\chi) + 1.\]
Since $\Ind_{K_0(\pp^m)}^{K} \psi_{\chi} \cong \bigoplus_{\ell = c(\chi)}^{m} \HH_{\chi,\ell}(S^{n - 1})$ from \eqref{eqn:HtoInd}, we conclude that $\End_{K}(\HH_{\chi,m}(S^{n - 1}))$ is one-dimensional by induction, and hence that $\HH_{\chi,m}(S^{n - 1})$ is irreducible. Finally, the orthogonal decompositions \eqref{eqn:CchiKmdecomp} and \eqref{eqn:Cdecomp} are clear via \eqref{eqn:CKmdecomp} and the fact that the union of the spaces $C^{\infty}(S^{n - 1})^{K(\pp^m)}$ is equal in $C^{\infty}(S^{n - 1})$.
\end{proof}

\begin{remark}
The key input to the proof of \hyperref[thm:Hchimirred]{Theorem \ref*{thm:Hchimirred}} is the fact that the double coset space $K_0(\pp^m) \backslash K / K_0(\pp^m)$ has cardinality $m + 1$, which is a consequence of \hyperref[lem:doublecoset]{Lemma \ref*{lem:doublecoset}}. We sketch below a more geometric proof of this fact. First, we note that this double coset space is also equal to $K_1(\pp^m) \backslash K / K_0(\pp^m)$. Since these groups all contain $K(\pp^m)$ as a normal subgroup, the cardinality of this double coset space is the same as that of
\[(K_0(\pp^m) / K(\pp^m)) \backslash (K/K(\pp^m)) / (K_0(\pp^m) / K(\pp^m)),\]
which in turn is the same as that of
\[\Pgp_n(\OO/\pp^m) \backslash \GL_n(\OO/\pp^m) / \Pgp_{(n - 1,1)}(\OO/\pp^m).\]
Here $\Pgp_n$ denotes the mirabolic subgroup \eqref{eqn:mirabolic} while $\Pgp_{(n - 1,1)}$ denotes the standard maximal parabolic subgroup of type $(n - 1,1)$. As in \eqref{eqn:barSn-1}, we identify $\Pgp_n(\OO/\pp^m) \backslash \GL_n(\OO/\pp^m)$ with $\overline{S}^{n - 1}$. Since $K_{n - 1}$ acts transitively on $S^{n - 2}$, we deduce that the action of $\Pgp_{(n - 1,1)}(\OO/\pp^m)$ on $\overline{S}^{n - 1}$ has $m + 1$ orbits, which are given by
\[\left\{\overline{x} = (\overline{x_1},\ldots,\overline{x_n}) \in \overline{S}^{n - 1} : \max\{|\overline{x_1}|,\ldots,|\overline{x_{n - 1}}|\} = q^{-j}\right\}\]
for $j \in \{0,\ldots,m\}$.
\end{remark}

\section{Zonal Spherical Functions}
\label{sect:zonal}

Let
\[C^{\infty}(S^{n - 1})^{K_{n - 1,1}} \coloneqq \left\{f \in C^{\infty}(S^{n - 1}) : f(xk') = f(x) \text{ for all $x \in S^{n - 1}$ and $k' \in K_{n - 1,1}$}\right\}\]
denote the subspace of $K_{n - 1,1}$-invariant locally constant functions on the unit sphere. We identify precisely which elements of this lie in $\HH_{\chi,m}(S^{n - 1})$ for each character $\chi$ of $\OO^{\times}$ and nonnegative integer $m \geq c(\chi)$.

\begin{lemma}
The subspace $C^{\infty}(S^{n - 1})^{K_{n - 1,1}} \cap C^{\infty}(S^{n - 1})_{\chi}^{K(\pp^m)}$ of $C^{\infty}(S^{n - 1})_{\chi}^{K(\pp^m)}$ has dimension $m - c(\chi) + 1$ and is spanned by the functions
\begin{equation}
\label{eqn:phi}
\phi_{\chi,\ell}(x_1,\ldots,x_n) \coloneqq \begin{dcases*}
\chi(x_n) & if $\max\{|x_1|,\ldots,|x_{n - 1}|\} \leq q^{-\ell}$,	\\
0 & if $q^{-\ell} < \max\{|x_1|,\ldots,|x_{n - 1}|\} \leq 1$,
\end{dcases*}
\end{equation}
for $\ell \in \{c(\chi),\ldots,m\}$. Furthermore, for $\ell_1,\ell_2 \geq c(\chi)$,
\begin{equation}
\label{eqn:phiinnerproduct}
\left\langle \phi_{\chi,\ell_1}, \phi_{\chi,\ell_2}\right\rangle = \begin{dcases*}
1 & if $\ell_1 = \ell_2 = 0$,	\\
\frac{q - 1}{q^{(\max\{\ell_1,\ell_2\} - 1)(n - 1)} (q^n - 1)} & if $\max\{\ell_1,\ell_2\} \geq 1$.
\end{dcases*}
\end{equation}
\end{lemma}

\begin{proof}
That the dimension of this subspace is $m - c(\chi) + 1$ is a direct consequence of \eqref{eqn:CchiKmdecomp} and \hyperref[cor:onedim]{Corollary \ref*{cor:onedim}}. It is then straightforward to see that $\phi_{\chi,\ell}$ is an element of this subspace for each $\ell \in \{c(\chi),\ldots,m\}$ and that these are linearly independent. Finally, the identity \eqref{eqn:phiinnerproduct} is immediate from the definition \eqref{eqn:phi} of $\phi_{\chi,\ell}$, for this implies that
\[\left\langle \phi_{\chi,\ell_1}, \phi_{\chi,\ell_2}\right\rangle = \vol(K_0(\pp^{\max\{\ell_1,\ell_2\}})) = \frac{1}{\left[K : K_0(\pp^{\max\{\ell_1,\ell_2\}}) \right]},\]
which is precisely the right-hand side of \eqref{eqn:phiinnerproduct}.
\end{proof}

It is clear that $\phi_{\chi,c(\chi)} \in \HH_{\chi,c(\chi)}(S^{n - 1})^{K_{n - 1,1}}$ and that $\phi_{\chi,c(\chi)}(e_n) = 1$. We show that there exist similar elements $P_{\chi,m}^{\circ}$ in $\HH_{\chi,m}(S^{n - 1})^{K_{n - 1,1}}$ for each $m \geq c(\chi)$, which must be a linear combination of the functions $\phi_{\chi,c(\chi)},\ldots,\phi_{\chi,m}$. We deduce the precise linear combination by making use of the fact that $P_{\chi,m}^{\circ}$ is orthogonal to $P_{\chi,j}^{\circ}$ whenever $j \in \{c(\chi),\ldots,m - 1\}$.

\begin{proposition}
\label{prop:Pcirc}
For each $\chi \in \widehat{\OO^{\times}}$ and $m \geq c(\chi)$, there exists a unique locally constant function $P_{\chi,m}^{\circ}$ in $\HH_{\chi,m}(S^{n - 1})$ satisfying $P_{\chi,m}^{\circ}(xk') = P_{\chi,m}^{\circ}(x)$ for all $x \in S^{n - 1}$ and $k' \in K_{n - 1,1}$ and $P_{\chi,m}^{\circ}(e_n) = 1$. This function is given by
\begin{equation}
\label{eqn:Pchimcircdef}
P_{\chi,m}^{\circ}(x_1,\ldots,x_n) = \begin{dcases*}
\chi(x_n) & if $\max\{|x_1|,\ldots,|x_{n - 1}|\} \leq q^{-m}$,	\\
\alpha_{\chi,m;m - 1} \chi(x_n) & if $m > c(\chi)$ and $\max\{|x_1|,\ldots,|x_{n - 1}|\} = q^{-m + 1}$,	\\
0 & otherwise,
\end{dcases*}
\end{equation}
where
\begin{equation}
\label{eqn:alphachimm-1}
\alpha_{\chi,m;m - 1} = \begin{dcases*}
- \frac{q - 1}{q(q^{n - 1} - 1)} & if $m = 1$,	\\
- \frac{1}{q^{n - 1} - 1} & if $m \geq 2$.
\end{dcases*}
\end{equation}
In particular, for all $k \in K$, we have that
\begin{equation}
\label{eqn:k-1bar}
P_{\chi,m}^{\circ}(e_n k) = \overline{P_{\chi,m}^{\circ}}(e_n k^{-1}).
\end{equation}
\end{proposition}

\begin{definition}
We call $P_{\chi,m}^{\circ}$ the zonal spherical function on $S^{n - 1}$ of character $\chi$ and level $m$.
\end{definition}

\begin{proof}[Proof of {\hyperref[prop:Pcirc]{Proposition \ref*{prop:Pcirc}}}]
Since $\HH_{\chi,m}(S^{n - 1})^{K_{n - 1,1}}$ is one-dimensional from \hyperref[cor:onedim]{Corollary \ref*{cor:onedim}}, there is a unique function $P_{\chi,m}^{\circ} \in \HH_{\chi,m}(S^{n - 1})^{K_{n - 1,1}}$ satisfying $P_{\chi,m}^{\circ}(e_n) = 1$. From \eqref{eqn:phi}, there exist constants $\alpha_{\chi,m;\ell} \in \C$ for $\ell \in \{c(\chi),\ldots,m\}$ such that
\begin{align*}
P_{\chi,m}^{\circ}(x_1,\ldots,x_n) & = \sum_{\ell = c(\chi)}^{m} \alpha_{\chi,m;\ell} \varphi_{\chi,m;\ell}(x)	\\
& = \begin{dcases*}
\alpha_{\chi,m;m} \chi(x_n) & if $\max\{|x_1|,\ldots,|x_{n - 1}|\} \leq q^{-m}$,	\\
\alpha_{\chi,m;\ell} \chi(x_n) & if $q^{-m} < \max\{|x_1|,\ldots,|x_{n - 1}|\} = q^{-\ell} \leq q^{-c(\chi)}$,	\\
0 & if $q^{-c(\chi)} < \max\{|x_1|,\ldots,|x_{n - 1}|\} \leq 1$,
\end{dcases*}
\end{align*}
where
\begin{align*}
\varphi_{\chi,m;\ell}(x_1,\ldots,x_n) & \coloneqq \begin{dcases*}
\phi_{\chi,\ell}(x_1,\ldots,x_n) - \phi_{\chi,\ell + 1}(x_1,\ldots,x_n) & for $\ell \in \{c(\chi),\ldots,m - 1\}$,	\\
\phi_{\chi,m}(x_1,\ldots,x_n) & for $\ell = m$,
\end{dcases*}	\\
& = \begin{dcases*}
\chi(x_n) & if $\ell \in \{c(\chi),\ldots, m - 1\}$ and $\max\{|x_1|,\ldots,|x_{n - 1}|\} = q^{-\ell}$,	\\
\chi(x_n) & if $\ell = m$ and $\max\{|x_1|,\ldots,|x_{n - 1}|\} \leq q^{-m}$,	\\
0 & otherwise.
\end{dcases*}
\end{align*}

Since $P_{\chi,m}^{\circ}(e_n) = 1$, we have that $\alpha_{\chi,m;m} = 1$. To deduce properties of the remaining coefficients, we make use the fact that
\begin{equation}
\label{eqn:betachimell}
\beta_{\chi,m;\ell} \coloneqq \left\langle \varphi_{\chi,m;\ell}, \varphi_{\chi,m;\ell}\right\rangle = \begin{dcases*}
\left\langle \phi_{\chi,\ell},\phi_{\chi,\ell}\right\rangle - \left\langle \phi_{\chi,\ell + 1}, \phi_{\chi,\ell + 1}\right\rangle & if $\ell \in \{c(\chi),\ldots,m - 1\}$,	\\
\left\langle \phi_{\chi,m}, \phi_{\chi,m}\right\rangle & if $\ell = m$	\\
\end{dcases*}
\end{equation}
from \eqref{eqn:phiinnerproduct}, while
\begin{equation}
\label{eqn:ell1neqell2}
\left\langle \varphi_{\chi,m;\ell_1}, \varphi_{\chi,m;\ell_2}\right\rangle = 0 \ \text{ if $\ell_1 \neq \ell_2$}
\end{equation}
since these have disjoint support. The zonal spherical function $P_{\chi,m}^{\circ}$ is orthogonal to $P_{\chi,j}^{\circ}$ for all $j \in \{c(\chi),\ldots,m - 1\}$ as $\HH_{\chi,m}(S^{n - 1})$ and $\HH_{\chi,j}(S^{n - 1})$ are mutually orthogonal. Writing
\[P_{\chi,j}^{\circ}(x_1,\ldots,x_n) = \sum_{\ell = c(\chi)}^{j - 1} \alpha_{\chi,j;\ell} \varphi_{\chi,m;\ell}(x_1,\ldots,x_n) + \sum_{\ell = j}^{m} \varphi_{\chi,m;\ell}(x_1,\ldots,x_n),\]
we deduce from \eqref{eqn:betachimell} and \eqref{eqn:ell1neqell2} that for each $j \in \{c(\chi),\ldots,m - 1\}$,
\begin{equation}
\label{eqn:innerproductorth}
\left\langle P_{\chi,m}^{\circ}, P_{\chi,j}^{\circ}\right\rangle = \sum_{\ell = c(\chi)}^{j - 1} \alpha_{\chi,m;\ell} \overline{\alpha_{\chi,j;\ell}} \beta_{\chi,m;\ell} + \sum_{\ell = j}^{m} \alpha_{\chi,m;\ell} \beta_{\chi,m;\ell} = 0.
\end{equation}

Using \eqref{eqn:innerproductorth}, we shall prove by induction that
\begin{equation}
\label{eqn:alphainduct}
\alpha_{\chi,m;\ell} = \begin{dcases*}
1 & if $\ell = m$,	\\
-\frac{\beta_{\chi,m;m}}{\beta_{\chi,m;m - 1}} & if $m > c(\chi)$ and $\ell = m - 1$,	\\
0 & otherwise,
\end{dcases*}
\end{equation}
which yields \eqref{eqn:Pchimcircdef}; the identity \eqref{eqn:alphachimm-1} then follows from \eqref{eqn:phiinnerproduct} and \eqref{eqn:betachimell}. The base case of \eqref{eqn:alphainduct} is $m = c(\chi)$, so that $\ell = c(\chi)$, in which case we have that $\alpha_{\chi,c(\chi);c(\chi)} = 1$ since $P_{\chi,c(\chi)}^{\circ}(e_n) = 1$. Now we suppose that \eqref{eqn:alphainduct} holds with $m$ replaced by $j$ for each $j \in \{c(\chi),\ldots,m - 1\}$. From \eqref{eqn:innerproductorth}, we have by the induction hypothesis that for $m \geq c(\chi) + 2$,
\[\left\langle P_{\chi,m}^{\circ}, P_{\chi,c(\chi)}^{\circ}\right\rangle - \left\langle P_{\chi,m}^{\circ}, P_{\chi,c(\chi) + 1}^{\circ}\right\rangle = \alpha_{\chi,m;c(\chi)} \beta_{\chi,m;c(\chi)} \left(1 - \frac{\beta_{\chi,c(\chi) + 1;c(\chi) + 1}}{\beta_{\chi,c(\chi) + 1;c(\chi)}}\right) = 0.\]
Thus $\alpha_{\chi,m;c(\chi)} = 0$. Proceeding inductively, we conclude that for all $j \in \{c(\chi),\ldots, m - 2\}$,
\[\left\langle P_{\chi,m}^{\circ}, P_{\chi,j}^{\circ}\right\rangle - \left\langle P_{\chi,m}^{\circ}, P_{\chi,j + 1}^{\circ}\right\rangle = \alpha_{\chi,m;j} \beta_{\chi,m;j} \left(1 - \frac{\beta_{\chi,j + 1;j + 1}}{\beta_{\chi,j + 1;j}}\right) = 0,\]
so that $\alpha_{\chi,m;j} = 0$. Finally, for $m \geq c(\chi) + 1$, we have that
\[\left\langle P_{\chi,m}^{\circ}, P_{\chi,m - 1}^{\circ}\right\rangle = \alpha_{\chi,m;m - 1} \beta_{\chi,m;m - 1} + \beta_{\chi,m;m} = 0,\]
which yields the remaining case $\ell = m - 1$ of \eqref{eqn:alphainduct}.

It remains to prove \eqref{eqn:k-1bar}. It suffices to show that for all $k \in K$,
\begin{equation}
\label{eqn:phik-1bar}
\phi_{\chi,\ell}(e_n k) = \overline{\phi_{\chi,\ell}}(e_n k^{-1})
\end{equation}
for each nonnegative integer $\ell \geq c(\chi)$, since $P_{\chi,m}^{\circ}$ may be written as a linear combination of such functions. The identity \eqref{eqn:phik-1bar} is clear if $k \notin K_0(\pp^{\ell})$, for then both sides are zero. If $k \in K_0(\pp^{\ell})$, so that $k^{-1} \in K_0(\pp^{\ell})$ as well, then $e_n k \prescript{t}{}{e}_n e_n k^{-1} \prescript{t}{}{e}_n - 1 \in \pp^{\ell}$ as $k k^{-1} = 1_n$, and consequently $\chi(e_n k \prescript{t}{}{e}_n) = \overline{\chi}(e_n k^{-1} \prescript{t}{}{e}_n)$ since $\ell \geq c(\chi)$, which yields the result.
\end{proof}

The zonal spherical function $P_{\chi,m}^{\circ}$ is useful for understanding various properties of the space $\HH_{\chi,m}(S^{n - 1})$. Our first application is the following lemma, which may be thought of as the addition theorem for $\HH_{\chi,m}(S^{n - 1})$.

\begin{lemma}
\label{lem:addthm}
Let $\{Q_j\}$ be an orthonormal basis of $\HH_{\chi,m}(S^{n - 1})$. Then for any $x \in S^{n - 1}$ and $k \in K$, we have that
\begin{equation}
\label{eqn:addthm}
\sum_{j = 1}^{\dim \HH_{\chi,m}(S^{n - 1})} Q_j(x) \overline{Q_j}(e_n k) = \dim \HH_{\chi,m}(S^{n - 1}) P_{\chi,m}^{\circ}(xk^{-1}).
\end{equation}
\end{lemma}

\begin{proof}
Since
\begin{equation}
\label{eqn:Qexpansion}
Q(xk) = (\tau_{\chi,m}(k) \cdot Q)(x) = \sum_{j = 1}^{\dim \HH_{\chi,m}(S^{n - 1})} \left\langle \tau_{\chi,m}(k) \cdot Q, Q_j \right\rangle Q_j(x)
\end{equation}
for any $Q \in \HH_{\chi,m}(S^{n - 1})$ and $x \in S^{n - 1}$, we have that for any $k \in K$,
\begin{equation}
\label{eqn:delta12}
\begin{split}
\delta_{\ell_1, \ell_2} & = \left\langle Q_{\ell_1}, Q_{\ell_2}\right\rangle	\\
& = \left\langle \tau_{\chi,m}(k^{-1}) \cdot Q_{\ell_1}, \tau_{\chi,m}(k^{-1}) \cdot Q_{\ell_2}\right\rangle	\\
& = \sum_{j_1, j_2 = 1}^{\dim \HH_{\chi,m}(S^{n - 1})} \left\langle \tau_{\chi,m}(k^{-1}) \cdot Q_{\ell_1}, Q_{j_1}\right\rangle \left\langle Q_{j_2}, \tau_{\chi,m}(k^{-1}) \cdot Q_{\ell_2}\right\rangle \left\langle Q_{j_1}, Q_{j_2}\right\rangle	\\
& = \sum_{j = 1}^{\dim \HH_{\chi,m}(S^{n - 1})} \left\langle \tau_{\chi,m}(k^{-1}) \cdot Q_{\ell_1}, Q_j\right\rangle \left\langle Q_j, \tau_{\chi,m}(k^{-1}) \cdot Q_{\ell_2}\right\rangle	\\
& = \sum_{j = 1}^{\dim \HH_{\chi,m}(S^{n - 1})} \left\langle Q_{\ell_1}, \tau_{\chi,m}(k) \cdot Q_j\right\rangle \left\langle \tau_{\chi,m}(k) \cdot Q_j, Q_{\ell_2}\right\rangle.
\end{split}
\end{equation}
For fixed $k \in K$, we now define $\QQ_{e_n k} \in \HH_{\chi,m}(S^{n - 1})$ by
\[\QQ_{e_n k}(x) \coloneqq \sum_{j = 1}^{\dim \HH_{\chi,m}(S^{n - 1})} Q_j(x) \overline{Q_j}(e_n k).\]
Then from \eqref{eqn:Qexpansion} and \eqref{eqn:delta12}, we have that for all $k_1 \in K$,
\begin{equation}
\label{eqn:QQ}
\begin{split}
\QQ_{e_n k k_1}(xk_1) & = \sum_{j = 1}^{\dim \HH_{\chi,m}(S^{n - 1})} Q_j(xk_1) \overline{Q_j}(e_n k k_1)	\\
& = \sum_{j, \ell_1, \ell_2 = 1}^{\dim \HH_{\chi,m}(S^{n - 1})} \left\langle \tau_{\chi,m}(k_1) \cdot Q_j, Q_{\ell_1}\right\rangle \left\langle Q_{\ell_2}, \tau_{\chi,m}(k_1) \cdot Q_j\right\rangle Q_{\ell_1}(x) \overline{Q_{\ell_2}}(e_n k)	\\
& = \sum_{\ell = 1}^{\dim \HH_{\chi,m}(S^{n - 1})} Q_{\ell}(x) \overline{Q_{\ell}}(e_n k)	\\
& = \QQ_{e_n k}(x).
\end{split}
\end{equation}
We use \eqref{eqn:QQ} with $x$ replaced by $xk$ and $k_1$ replaced by $k^{-1} k' k$ for $k' \in K_{n - 1,1}$. Since $e_n k' = e_n$, we deduce that
\[\left(\tau_{\chi,m}(k') \cdot \left(\tau_{\chi,m}(k) \cdot \QQ_{e_n k}\right)\right)(x) = \QQ_{e_n k}(x k' k) = \QQ_{e_n k}(xk) = \left(\tau_{\chi,m}(k) \cdot \QQ_{e_n k}\right)(x)\]
for all $k' \in K_{n - 1,1}$. As $\HH_{\chi,m}(S^{n - 1})^{K_{n - 1,1}}$ is one-dimensional, it follows that $\tau_{\chi,m}(k) \cdot \QQ_{e_n k}$ must be a constant multiple of $P_{\chi,m}^{\circ}$, and this constant is readily seen to be $\QQ_{e_n k}(e_n k)$ upon taking $x = e_n$. So for all $x \in S^{n - 1}$ and $k \in K$,
\begin{align*}
\sum_{j = 1}^{\dim \HH_{\chi,m}(S^{n - 1})} Q_j(x) \overline{Q_j}(e_n k) & = \QQ_{e_n k}(x)	\\
& = \left(\tau_{\chi,m}(k) \cdot \QQ_{e_n k}\right)(xk^{-1})	\\
& = \QQ_{e_n k}(e_n k) P_{\chi,m}^{\circ}(xk^{-1})	\\
& = \sum_{j = 1}^{\dim \HH_{\chi,m}(S^{n - 1})} \left|Q_j(e_n k)\right|^2 P_{\chi,m}^{\circ}(xk^{-1}).
\end{align*}

It remains to show that for all $k \in K$,
\begin{equation}
\label{eqn:constant}
\sum_{j = 1}^{\dim \HH_{\chi,m}(S^{n - 1})} \left|Q_j(e_n k)\right|^2 = \dim \HH_{\chi,m}(S^{n - 1}).
\end{equation}
To prove \eqref{eqn:constant}, we take $x = e_n k$ and $k_1 = k^{-1}$ in \eqref{eqn:QQ} in order to see that
\[\sum_{j = 1}^{\dim \HH_{\chi,m}(S^{n - 1})} \left|Q_j(e_n k)\right|^2 = \sum_{j = 1}^{\dim \HH_{\chi,m}(S^{n - 1})} \left|Q_j(e_n)\right|^2\]
for all $k \in K$ and hence the left-hand side is a constant. Integrating over $K \ni k$, we find that
\begin{align*}
\sum_{j = 1}^{\dim \HH_{\chi,m}(S^{n - 1})} |Q_j(e_n k)|^2 & = \int_{K} \sum_{j = 1}^{\dim \HH_{\chi,m}(S^{n - 1})} |Q_j(e_n k)|^2 \, dk	\\
& = \sum_{j = 1}^{\dim \HH_{\chi,m}(S^{n - 1})} \left\langle Q_j, Q_j\right\rangle	\\
& = \dim \HH_{\chi,m}(S^{n - 1}).
\qedhere
\end{align*}
\end{proof}

A simple consequence of \hyperref[lem:addthm]{Lemma \ref*{lem:addthm}} is the following result showing that certain matrix coefficients of $\tau_{\chi,m}$ are themselves elements of $\HH_{\chi,m}(S^{n - 1})$.

\begin{corollary}
\label{cor:reproducing}
The reproducing kernel for $\HH_{\chi,m}(S^{n - 1})$ is $(\dim \HH_{\chi,m}(S^{n - 1})) \tau_{\chi,m}(k^{-1}) \cdot P_{\chi,m}^{\circ}$, so that for all $P \in \HH_{\chi,m}(S^{n - 1})$ and $k \in K$,
\begin{equation}
\label{eqn:reproducing}
P(e_n k) = \dim \HH_{\chi,m}(S^{n - 1}) \left\langle \tau_{\chi,m}(k) \cdot P, P_{\chi,m}^{\circ}\right\rangle.
\end{equation}
In particular,
\begin{equation}
\label{eqn:PchimcircL2}
\left\langle P_{\chi,m}^{\circ}, P_{\chi,m}^{\circ} \right\rangle = \frac{1}{\dim \HH_{\chi,m}(S^{n - 1})}.
\end{equation}
\end{corollary}

One can also show \eqref{eqn:PchimcircL2} in a more direct fashion by combining \eqref{eqn:phiinnerproduct}, \eqref{eqn:Pchimcircdef}, and \eqref{eqn:alphachimm-1}.

\begin{proof}
From \eqref{eqn:addthm} and \eqref{eqn:Qexpansion}, we have that
\begin{align*}
P(e_n k) & = \sum_{j = 1}^{\dim \HH_{\chi,m}(S^{n - 1})} \left\langle P, Q_j\right\rangle Q_j(e_n k)	\\
& = \left\langle P, \sum_{j = 1}^{\dim \HH_{\chi,m}(S^{n - 1})} Q_j \overline{Q_j}(e_n k)\right\rangle	\\
& = \left\langle P, \dim \HH_{\chi,m}(S^{n - 1})\tau_{\chi,m}(k^{-1}) \cdot P_{\chi,m}^{\circ}\right\rangle	\\
& = \dim \HH_{\chi,m}(S^{n - 1}) \left\langle \tau_{\chi,m}(k) \cdot P, P_{\chi,m}^{\circ}\right\rangle.
\end{align*}
The identity \eqref{eqn:PchimcircL2} then follows upon taking $P = P_{\chi,m}^{\circ}$ and $k = 1_n$ since $P_{\chi,m}^{\circ}(e_n) = 1$.
\end{proof}

\section{The Newform \texorpdfstring{$K$}{K}-Type}
\label{sect:Ktype}

\subsection{The Newform and the Conductor Exponent}

Let $(\pi,V_{\pi})$ be an induced representation of Langlands type of $\GL_n(F)$. Thus there exist positive integers $n_1,\ldots,n_r$ for which $n_1 + \cdots + n_r = n$ and essentially square-integrable representations $(\pi_j,V_{\pi_j})$ of $\GL_{n_j}(F)$ of the form $\sigma_j \otimes \left|\det\right|^{t_j}$, where $\sigma_j$ is square-integrable and $t_j \in \C$ satisfies $\Re(t_1) \geq \cdots \geq \Re(t_r)$, such that
\[\pi = \Ind_{\Pgp(F)}^{\GL_n(F)} \bigboxtimes_{j = 1}^{r} \pi_j,\]
the representation obtained by normalised parabolic induction from the standard upper parabolic subgroup $\Pgp(F) = \Pgp_{(n_1,\ldots,n_r)}(F)$ of $\GL_n(F)$. One can take as a model for $\pi$ the space of smooth functions $f : \GL_n(F) \to V_{\pi_1} \otimes \cdots \otimes V_{\pi_r}$, upon which $\pi$ acts via right translations, that satisfy
\[f(umg) = \prod_{j = 1}^{r} \left|\det m_j\right|^{\frac{1}{2} (n - 2(n_1 + \cdots + n_{j - 1}) - n_j)} \bigotimes_{j = 1}^{r} \pi_j(m_j) \cdot f(g)\]
for all $u \in \Ngp_{\Pgp}(F)$, the unipotent radical of $\Pgp(F)$, $m = \blockdiag(m_1,\ldots,m_r) \in \Mgp_{\Pgp}(F)$, the Levi subgroup of $\Pgp(F)$, and $g \in \GL_n(F)$. We note that every generic irreducible admissible smooth representation of $\GL_n(F)$ is isomorphic to some induced representation of Langlands type; see, for example, \cite{JS83}.

For each nonnegative integer $m$ and each character $\chi \in \widehat{\OO^{\times}}$ for which $0 \leq c(\chi) \leq m$, we define the subspaces
\begin{align*}
V_{\pi}^{K_1(\pp^m)} & \coloneqq \left\{v \in V_{\pi} : \pi(k) \cdot v = v \text{ for all } k \in K_1(\pp^m)\right\},	\\
V_{\pi}^{K_0(\pp^m), \chi} & \coloneqq \left\{v \in V_{\pi} : \pi(k) \cdot v = \psi_{\chi}(k) v \text{ for all } k \in K_0(\pp^m)\right\},
\end{align*}
where $\psi_{\chi}$ denotes the character of $K_0(\pp^m)$ corresponding to $\chi$. The former is the subspace of $V_{\pi}$ of $K_1(\pp^m)$-invariant vectors; the latter is the subspace of $(K_0(\pp^m),\psi_{\chi})$-equivariant vectors.

\begin{lemma}
Let $(\pi,V_{\pi})$ be an induced representation of Langlands type of $\GL_n(F)$. For each nonnegative integer $m$, the subspace $V_{\pi}^{K_1(\pp^m)}$ is the image of the projection map $\Pi^{K_1(\pp^m)} : V_{\pi} \to V_{\pi}$ given by
\begin{equation}
\label{eqn:PiK1}
\Pi^{K_1(\pp^m)}(v) \coloneqq \frac{1}{\vol(K_1(\pp^m))} \int_{K_1(\pp^m)} \pi(k) \cdot v \, dk,
\end{equation}
while $V_{\pi}^{K_0(\pp^m), \chi}$ is the image of the projection map $\Pi^{K_0(\pp^m), \chi} : V_{\pi} \to V_{\pi}$ given by
\begin{equation}
\label{eqn:PiK0}
\Pi^{K_0(\pp^m), \chi}(v) \coloneqq \begin{dcases*}
\int_{K} \pi(k) \cdot v \, dk & if $m = c(\chi) = 0$,	\\
\frac{1}{\vol(K_0(\pp^m))} \int_{K_0(\pp^m)} \overline{\chi}(e_n k \prescript{t}{}{e}_n) \pi(k) \cdot v \, dk & if $m > 0$.
\end{dcases*}
\end{equation}
Finally, we have that
\begin{equation}
\label{eqn:VpiK0toK1}
V_{\pi}^{K_0(\pp^m), \chi} = \begin{dcases*}
V_{\pi}^{K_1(\pp^m)} & if $\chi = \chi_{\pi}$,	\\
\{0\} & otherwise,
\end{dcases*}
\end{equation}
where $\chi_{\pi} \coloneqq \omega_{\pi}|_{\OO^{\times}}$ with $\omega_{\pi} : F^{\times} \to \C^{\times}$ the central character of $\pi$.
\end{lemma}

\begin{proof}
The two identities $\Pi^{K_1(\pp^m)}(V_{\pi}) = V_{\pi}^{K_1(\pp^m)}$ and $\Pi^{K_0(\pp^m), \chi}(V_{\pi}) = V_{\pi}^{K_0(\pp^m), \chi}$ are clear. Since $\Zgp(\OO) K_1(\pp^m) = K_0(\pp^m)$ and $\pi|_{\Zgp(\OO)} = \omega_{\pi}|_{\OO^{\times}} = \chi_{\pi}$, the identities $V_{\pi}^{K_0(\pp^m), \chi_{\pi}} = V_{\pi}^{K_1(\pp^m)}$ and $V_{\pi}^{K_0(\pp^m), \chi} = \{0\}$ for $\chi \neq \chi_{\pi}$ then follow.
\end{proof}

A fundamental result concerning induced representations of Langlands type is the existence of a newform. Jacquet, Piatetski-Shapiro, and Shalika \cite{JP-SS81} have shown that each induced representation of Langlands type $(\pi,V_{\pi})$ of $\GL_n(F)$ contains a distinguished vector $v^{\circ} \in V_{\pi}$, the newform, whose complexity is measured in a natural way by a nonnegative integer $c(\pi)$, the conductor exponent. This generalises a result of Casselman \cite{Cas73a}, who proved this for $n = 2$, and observed that when $F = \Q_p$ and $\pi$ is the local component of an automorphic representation of $\GL_2(\A_{\Q})$, this is the ad\`{e}lic reformulation of the Atkin--Lehner theory of newforms for classical modular forms \cite{AL70}.

\begin{theorem}[{\cite[Th\'{e}or\`{e}me (5)]{JP-SS81}}]
\label{thm:JP-SS}
Let $(\pi,V_{\pi})$ be an induced representation of Langlands type of $\GL_n(F)$. There exists a minimal nonnegative integer $m = c(\pi)$ for which the space $V_{\pi}^{K_1(\pp^m)} = V_{\pi}^{K_0(\pp^m), \chi_{\pi}}$ is nontrivial, in which case it is one-dimensional.
\end{theorem}

\begin{definition}
\label{def:newform}
The nonzero vector $v^{\circ} \in V_{\pi}^{K_1(\pp^{c(\pi)})} = V_{\pi}^{K_0(\pp^{c(\pi)}), \chi_{\pi}}$, unique up to scalar multiplication, is called the newform of $\pi$. The nonnegative integer $c(\pi)$ is called the conductor exponent of $\pi$. Elements of $V_{\pi}^{K_1(\pp^m)} = V_{\pi}^{K_0(\pp^m), \chi_{\pi}}$ for $m > c(\pi)$ are called oldforms.
\end{definition}

\begin{remark}
The proof of \cite[Th\'{e}or\`{e}me (5)]{JP-SS81} contains a gap; correct proofs were later independently given by Jacquet \cite[Theorem 1]{Jac12} and Matringe \cite[Corollary 3.3]{Mat13}.
\end{remark}

\begin{remark}
\label{rem:othercharacterisation}
As proven in \cite{JP-SS81}, there are other ways to characterise the newform and the conductor exponent of $\pi$ instead of in terms of the subspace $V_{\pi}^{K_1(\pp^{c(\pi)})}$ of $K_1(\pp^{c(\pi)})$-invariant vectors in $V_{\pi}$. The conductor exponent $c(\pi)$ is precisely the nonnegative integer for which the epsilon factor $\e(s,\pi,\psi)$ associated to $\pi$ is of the form
\[\e(s,\pi,\psi) = \e\left(\frac{1}{2},\pi,\psi\right) q^{-c(\pi) \left(s - \frac{1}{2}\right)},\]
where $\psi : F \to \C^{\times}$ is an unramified additive character of $F$. The newform is such that when viewed in the Whittaker model $\WW(\pi,\psi)$ of $\pi$, it is the unique Whittaker function $W^{\circ}$ satisfying $W^{\circ}\left(g \begin{psmallmatrix} k' & 0 \\ 0 & 1 \end{psmallmatrix}\right) = W^{\circ}(g)$ for all $g \in \GL_n(F)$ and $k' \in K_{n - 1}$ that is a test vector for the local $\GL_n \times \GL_{n - 1}$ Rankin--Selberg integral whenever the second representation is unramified, so that
\[\int\limits_{\Ngp_{n - 1}(F) \backslash \GL_{n - 1}(F)} W^{\circ} \begin{pmatrix} g & 0 \\ 0 & 1 \end{pmatrix} W^{\prime \circ}(g) \left|\det g\right|^{s - \frac{1}{2}} \, dg = L(s,\pi \times \pi')\]
for all spherical induced representations of Langlands type $\pi'$ of $\GL_{n - 1}(F)$ with spherical Whittaker function $W^{\prime\circ} \in \WW(\pi',\overline{\psi})$ normalised such that $W^{\prime\circ}(1_{n - 1}) = 1$ \cite[Th\'{e}or\`{e}me (4)]{JP-SS81} (see additionally \cite{Jac12} and \cite[Corollary 3.3]{Mat13}).
\end{remark}

\subsection{The Newform \texorpdfstring{$K$}{K}-Type}

We shall show that the newform lies in a distinguished $K$-type of $\pi$. To begin, for each irreducible smooth representation $\tau \in \widehat{K}$, we define the projection map $\Pi^{\tau} : V_{\pi} \to V_{\pi}$ by
\[\Pi^{\tau}(v) \coloneqq \int_{K} \xi^{\tau}(k) \pi(k) \cdot v \, dk,\]
where
\[\xi^{\tau}(k) \coloneqq (\dim \tau) \Tr \tau(k^{-1})\]
is the elementary idempotent associated to $\tau$. The image of $V_{\pi}$ under $\Pi^{\tau}$ is the $\tau$-isotypic subspace $V_{\pi}^{\tau}$ of $V_{\pi}$, which is finite-dimensional since $\pi$ is admissible. We say that $\tau$ is a $K$-type of $\pi$ if $\Hom_{K}(\tau, \pi|_{K})$ is nontrivial, in which case $\dim V_{\pi}^{\tau} = \dim \tau \dim \Hom_{K}(\tau, \pi|_{K}) > 0$, and we call $\dim \Hom_{K}(\tau, \pi|_{K})$ the multiplicity of $\tau$ in $\pi$.

In general, the $K$-type decomposition of an induced representation of Langlands type $\pi$ is not known except in special cases. For $n = 2$, this follows from work of Casselman \cite{Cas73b}, Silberger \cite{Sil70}, and Hansen \cite{Han87}; for $n = 3$ and $\pi$ a principal series representation, this problem has been studied by Campbell and Nevins \cite{CN09,CN10} and Onn and Singla \cite{OS14}.

We do not attempt to determine the full $K$-type decomposition of $\pi$, which is undoubtedly challenging, since no explicit description of $\widehat{K}$ currently exists for $n > 2$. Rather, we study the multiplicity with which particular representations $\tau \in \widehat{K}$ appear in the $K$-type decomposition of $\pi$, namely those representations containing a $K_{n - 1,1}$-fixed vector.

To do so, we define the projection map $\Pi^{K_{n - 1,1}} : V_{\pi} \to V_{\pi}$ given by
\[\Pi^{K_{n - 1,1}}(v) \coloneqq \int_{K_{n - 1,1}} \pi(k') \cdot v \, dk',\]
where $dk'$ denotes the Haar probability measure on the compact group $K_{n - 1,1}$, so that the image of $\Pi^{K_{n - 1,1}}$ is the subspace of $K_{n - 1,1}$-invariant vectors. The composition of the two projections $\Pi^{\tau}$ and $\Pi^{K_{n - 1,1}}$ is the projection
\begin{equation}
\label{eqn:PitauKn-11}
\left(\Pi^{\tau,K_{n - 1,1}}\right)(v) \coloneqq \left(\Pi^{K_{n - 1,1}} \circ \Pi^{\tau}\right)(v) = \left(\Pi^{\tau} \circ \Pi^{K_{n - 1,1}}\right)(v) = \int_{K} \xi^{\tau,K_{n - 1,1}}(k) \pi(k) \cdot v \, dk
\end{equation}
onto the subspace of $K_{n - 1,1}$-invariant $\tau$-isotypic vectors
\[V_{\pi}^{\tau,K_{n - 1,1}} \coloneqq \left(\Pi^{\tau,K_{n - 1,1}}\right)(V_{\pi}) = \left\{v \in V_{\pi}^{\tau} : \pi(k') \cdot v = v \text{ for all $k' \in K_{n - 1,1}$}\right\}.\]
Here
\begin{equation}
\label{eqn:xitauKn-11}
\xi^{\tau,K_{n - 1,1}}(k) \coloneqq \int_{K_{n - 1,1}} \xi^{\tau}(k'k) \, dk' = \int_{K_{n - 1,1}} \xi^{\tau}(kk') \, dk'.
\end{equation}

Finally, for any nonnegative integer $m$, we set
\begin{align*}
V_{\pi}(m) & \coloneqq \bigoplus_{\substack{\tau \in \widehat{K} \\ c(\tau) = m}} V_{\pi}^{\tau},	\\
V_{\pi}(m)^{K_{n - 1,1}} & \coloneqq \bigoplus_{\substack{\tau \in \widehat{K} \\ c(\tau) = m}} V_{\pi}^{\tau,K_{n - 1,1}}.
\end{align*}
Here the level $c(\tau)$ of $\tau$ is defined to be the minimal nonnegative integer $m$ for which the kernel of $\tau$ contains $K(\pp^m)$, which is necessarily finite since $\tau$ is smooth. Thus $V_{\pi}(m)^{K_{n - 1,1}}$ is the subspace of $K_{n - 1,1}$-invariant vectors that are linear combinations of $\tau$-isotypic vectors for some $\tau \in \widehat{K}$ of level $c(\tau) = m$. We now show that the newform and the conductor exponent may be characterised in terms of $V_{\pi}(m)^{K_{n - 1,1}}$.

\begin{theorem}
\label{thm:newform}
Let $(\pi,V_{\pi})$ be an induced representation of Langlands type of $\GL_n(F)$. For any nonnegative integer $m$, we have that
\begin{equation}
\label{eqn:dimVKn-11}
\dim V_{\pi}(m)^{K_{n - 1,1}} = \begin{dcases*}
\binom{m - c(\pi) + n - 2}{n - 2} & if $m \geq c(\pi)$,	\\
0 & otherwise.
\end{dcases*}
\end{equation}
In particular, the minimal nonnegative integer $m$ for which $V_{\pi}(m)^{K_{n - 1,1}}$ is nontrivial is $m = c(\pi)$. Furthermore, we have that
\begin{equation}
\label{eqn:VKn-11}
V_{\pi}(m)^{K_{n - 1,1}} = \begin{dcases*}
V_{\pi}^{\tau_{\chi_{\pi},m},K_{n - 1,1}} & if $m \geq c(\pi)$,	\\
\{0\} & otherwise,
\end{dcases*}
\end{equation}
and that
\begin{equation}
\label{eqn:VpiK1mtoVpitaum}
V_{\pi}^{K_1(\pp^m)} = V_{\pi}^{K_0(\pp^m), \chi_{\pi}} = \bigoplus_{\ell = 0}^{m} V_{\pi}(\ell)^{K_{n - 1,1}} = \begin{dcases*}
\bigoplus_{\ell = c(\pi)}^{m} V_{\pi}^{\tau_{\chi_{\pi},\ell},K_{n - 1,1}} & if $m \geq c(\pi)$,	\\
\{0\} & otherwise.
\end{dcases*}
\end{equation}
In particular,
\begin{equation}
\label{eqn:VpiK1cpitoVpitaucpi}
V_{\pi}^{K_1(\pp^{c(\pi)})} = V_{\pi}^{K_0(\pp^{c(\pi)}), \chi_{\pi}} = V_{\pi}(c(\pi))^{K_{n - 1,1}} = V_{\pi}^{\tau_{\chi_{\pi},c(\pi)},K_{n - 1,1}}.
\end{equation}
\end{theorem}

\begin{definition}
We call $\tau_{\chi_{\pi},c(\pi)}$ the newform $K$-type of $\pi$.
\end{definition}

\begin{remark}
\label{rem:newformdef}
This gives alternative characterisations of the newform and conductor exponent of $\pi$: we may \emph{define} the conductor exponent $c(\pi)$ of $\pi$ to be the minimal nonnegative integer $m$ for which the space $V_{\pi}(m)^{K_{n - 1,1}} = V_{\pi}^{\tau_{\chi_{\pi},m},K_{n - 1,1}}$ is nontrivial, while we may \emph{define} the newform to be the nonzero vector, unique up to scalar multiplication, lying in $V_{\pi}(c(\pi))^{K_{n - 1,1}} = V_{\pi}^{\tau_{\chi_{\pi},c(\pi)},K_{n - 1,1}}$. Of course, this definition is somewhat circular in practice, since, as we shall shortly see, we use \hyperref[thm:JP-SS]{Theorem \ref*{thm:JP-SS}} (or rather a consequence thereof) in order to prove \hyperref[thm:newform]{Theorem \ref*{thm:newform}}.
\end{remark}

\hyperref[thm:newform]{Theorem \ref*{thm:newform}} gives additional information about the newform not immediately apparent from \hyperref[thm:JP-SS]{Theorem \ref*{thm:JP-SS}}: not only is the newform the unique vector, up to scalar multiplication, that is $K_1(\pp^{c(\pi)})$-invariant (or, equivalently, $(K_0(\pp^{c(\pi)}),\psi_{\chi_{\pi}})$-equivariant), it is also the unique vector, up to scalar multiplication, that is $K_{n - 1,1}$-invariant and $\tau_{\chi_{\pi},c(\pi)}$-isotypic; moreover, there are no nontrivial $K_{n - 1,1}$-invariant $\tau_{\chi,m}$-isotypic vectors with $m < c(\pi)$ or $\chi \neq \chi_{\pi}$. We have also shown that the space $V_{\pi}^{K_1(\pp^m)}$ of oldforms of level $m \geq c(\pi)$ decomposes into the direct sum of the subspaces of $K_{n - 1,1}$-invariant $\tau_{\chi_{\pi},\ell}$-isotypic vectors in $V_{\pi}$ with $\ell \in \{c(\pi),\ldots,m\}$, each of which has dimension $\binom{\ell - c(\pi) + n - 2}{n - 2}$ respectively.

Before we prove \hyperref[thm:newform]{Theorem \ref*{thm:newform}}, we return to the projection $\Pi^{\tau,K_{n - 1,1}}$ defined in \eqref{eqn:PitauKn-11} in terms of integration against the function $\xi^{\tau,K_{n - 1,1}}$ as in \eqref{eqn:xitauKn-11}. This function has a particularly simple description.

\begin{lemma}
\label{lem:xitauKn-11}
For $\tau \in \widehat{K}$ and $k \in K$, we have that
\[\xi^{\tau,K_{n - 1,1}}(k) = \begin{dcases*}
(\dim \tau_{\chi,m}) P_{\chi,m}^{\circ}(e_n k^{-1}) & if $\tau = \tau_{\chi,m}$ for some $\chi \in \widehat{\OO^{\times}}$ and $m \geq c(\chi)$,	\\
0 & otherwise.
\end{dcases*}\]
\end{lemma}

\begin{proof}
This follows from \hyperref[cor:onedim]{Corollaries \ref*{cor:onedim}} and \ref{cor:reproducing}.
\end{proof}

From \hyperref[lem:xitauKn-11]{Lemma \ref*{lem:xitauKn-11}}, we have the following simple consequences.

\begin{corollary}
We have that
\begin{align}
\label{eqn:xitoK1}
\sum_{\ell = 0}^{m} \sum_{\substack{\chi \in \widehat{\OO^{\times}} \\ 0 \leq c(\chi) \leq \ell}} \xi^{\tau_{\chi,\ell},K_{n - 1,1}}(k) & = \begin{dcases*}
\frac{1}{\vol(K_1(\pp^m))} & if $k \in K_1(\pp^m)$,	\\
0 & otherwise,
\end{dcases*}	\\
\label{eqn:xitoK0}
\sum_{\ell = c(\chi)}^{m} \xi^{\tau_{\chi,\ell},K_{n - 1,1}}(k) & = \begin{dcases*}
1 & if $m = c(\chi) = 0$ and $k \in K$,	\\
\frac{\overline{\chi}(e_n k \prescript{t}{}{e}_n)}{\vol(K_0(\pp^m))} & if $m > 0$ and $k \in K_0(\pp^m)$,	\\
0 & otherwise.
\end{dcases*}
\end{align}
Consequently, for an induced representation of Langlands type $(\pi,V_{\pi})$ of $\GL_n(F)$, we have that
\begin{align}
\label{eqn:PiK1tau}
\Pi^{K_1(\pp^m)} & = \sum_{\ell = 0}^{m} \sum_{\substack{\chi \in \widehat{\OO^{\times}} \\ 0 \leq c(\chi) \leq \ell}} \Pi^{\tau_{\chi,\ell},K_{n - 1,1}},	\\
\label{eqn:PiK0tau}
\Pi^{K_0(\pp^m), \chi} & = \sum_{\ell = c(\chi)}^{m} \Pi^{\tau_{\chi,\ell},K_{n - 1,1}}
\end{align}
for any nonnegative integer $m$ and any character $\chi \in \widehat{\OO^{\times}}$ for which $0 \leq c(\chi) \leq m$.
\end{corollary}

\begin{proof}
The identity \eqref{eqn:xitoK0} can be seen by combining \hyperref[prop:Pcirc]{Proposition \ref*{prop:Pcirc}} and \hyperref[lem:dimH]{Lemmata \ref*{lem:dimH}} and \ref{lem:xitauKn-11}, noting that $1/\vol(K_0(\pp^m)) = [K : K_0(\pp^m)]$. The identity \eqref{eqn:xitoK1} follows from \eqref{eqn:xitoK0} together with character orthogonality, namely the fact that for $x \in \OO^{\times}$,
\[\sum_{\substack{\chi \in \widehat{\OO^{\times}} \\ 0 \leq c(\chi) \leq m}} \chi(x) = \begin{dcases*}
\# \OO^{\times} / (1 + \pp^m) & if $x - 1 \in \pp^m$,	\\
0 & otherwise,
\end{dcases*}\]
and noting that $\# \OO^{\times} / (1 + \pp^m) = [K_0(\pp^m) : K_1(\pp^m)]$. The identities \eqref{eqn:PiK1tau} and \eqref{eqn:PiK0tau} for the projections $\Pi^{K_1(\pp^m)}$ and $\Pi^{K_0(\pp^m), \chi}$ in \eqref{eqn:PiK1} and \eqref{eqn:PiK0} in terms of the projections $\Pi^{\tau_{\chi,\ell},K_{n - 1,1}}$ in \eqref{eqn:PitauKn-11} are immediate consequences of \eqref{eqn:xitoK1} and \eqref{eqn:xitoK0}.
\end{proof}

\begin{proof}[Proof of {\hyperref[thm:newform]{Theorem \ref*{thm:newform}}}]
We first note that $V_{\pi}^{\tau_{\chi_1,m_1},K_{n - 1,1}}$ and $V_{\pi}^{\tau_{\chi_2,m_2},K_{n - 1,1}}$ are mutually orthogonal whenever either $\chi_1 \neq \chi_2$ or $m_1 \neq m_2$. In conjunction with \eqref{eqn:VpiK0toK1} and \eqref{eqn:PiK0tau}, this implies that
\[V_{\pi}^{K_1(\pp^m)} = V_{\pi}^{K_0(\pp^m), \chi_{\pi}} = \begin{dcases*}
\bigoplus_{\ell = c(\chi_{\pi})}^{m} V_{\pi}^{\tau_{\chi_{\pi},\ell},K_{n - 1,1}} & if $m \geq c(\chi_{\pi})$,	\\
\{0\} & otherwise.
\end{dcases*}\]
Since
\[V_{\pi}^{K_1(\pp^m)} = V_{\pi}^{K_0(\pp^m), \chi_{\pi}} = \{0\} \ \text{ whenever $m < c(\pi)$}\]
from \hyperref[thm:JP-SS]{Theorem \ref*{thm:JP-SS}}, we deduce that
\begin{equation}
\label{eqn:Vpitau=0}
V_{\pi}^{\tau_{\chi_{\pi},m},K_{n - 1,1}} = \{0\} \ \text{ whenever $c(\chi_{\pi}) \leq m < c(\pi)$},
\end{equation}
noting that $c(\chi_{\pi}) = c(\omega_{\pi}) \leq c(\pi)$, from which \eqref{eqn:VpiK1mtoVpitaum} and \eqref{eqn:VpiK1cpitoVpitaucpi} both follow.

Next, we have from \hyperref[lem:xitauKn-11]{Lemma \ref*{lem:xitauKn-11}} and the fact that $\pi|_{\Zgp(\OO)} = \chi_{\pi}$ that
\[V_{\pi}(m)^{K_{n - 1,1}} = \begin{dcases*}
V_{\pi}^{\tau_{\chi_{\pi},m},K_{n - 1,1}} & if $m \geq c(\chi_{\pi})$,	\\
\{0\} & otherwise.
\end{dcases*}\]
Together with \eqref{eqn:Vpitau=0}, we deduce \eqref{eqn:VKn-11}.

Finally, from \cite[Theorem 1]{Ree91} (which in turn relies on \hyperref[thm:JP-SS]{Theorem \ref*{thm:JP-SS}}), we have that
\[\dim V_{\pi}^{K_1(\pp^m)} = \dim V_{\pi}^{K_0(\pp^m), \chi_{\pi}} = \begin{dcases*}
\binom{m - c(\pi) + n - 1}{n - 1} & if $m \geq c(\pi)$,	\\
0 & otherwise.
\end{dcases*}\]
The identity \eqref{eqn:dimVKn-11} then follows from \eqref{eqn:VKn-11}, \eqref{eqn:VpiK1mtoVpitaum}, and induction together with the combinatorial identity
\begin{equation}
\label{eqn:binom}
\sum_{\ell = c(\pi)}^{m} \binom{\ell - c(\pi) + n - 2}{n - 2} = \binom{m - c(\pi) + n - 1}{n - 1}.
\qedhere
\end{equation}
\end{proof}

We may use \hyperref[thm:newform]{Theorem \ref*{thm:newform}} to tell us the multiplicity with which an irreducible smooth representations $\tau_{\chi,m} \in \widehat{K}$ occurs in the restriction of $\pi$ to $K$; when $\chi_{\pi}$ is trivial, this was observed by Reeder \cite{Ree94} to follow upon combining results from \cite{Hil94} and \cite{Ree91}.

\begin{proposition}
\label{prop:newform}
Let $(\pi,V_{\pi})$ be an induced representation of Langlands type of $\GL_n(F)$. For any nonnegative integer $m$ and any character $\chi \in \widehat{\OO^{\times}}$ for which $0 \leq c(\chi) \leq m$, we have that
\[\dim \Hom_{K}\left(\tau_{\chi,m},\pi|_{K}\right) = \begin{dcases*}
\binom{m - c(\pi) + n - 2}{n - 2} & if $m \geq c(\pi)$ and $\chi = \chi_{\pi}$,	\\
0 & otherwise.
\end{dcases*}\]
\end{proposition}

While this can be proved directly using \hyperref[thm:newform]{Theorem \ref*{thm:newform}}, we give an alternate proof via Frobenius reciprocity.

\begin{proof}
It is clear that $\Hom_{K}(\tau_{\chi,m},\pi|_{K})$ is trivial if $\chi \neq \chi_{\pi}$. For $\chi = \chi_{\pi}$,
by \eqref{eqn:HtoInd} and Frobenius reciprocity, we have that
\[\bigoplus_{\ell = c(\chi_{\pi})}^{m} \Hom_{K}\left(\tau_{\chi_{\pi},\ell}, \pi|_{K}\right) \cong \Hom_{K}\left(\Ind_{K_0(\pp^m)}^{K} \psi_{\chi_{\pi}}, \pi|_{K}\right) \cong \Hom_{K_0(\pp^m)}\left(\psi_{\chi_{\pi}}, \pi|_{K_0(\pp^m)}\right).\]
Since $\Hom_{K_0(\pp^m)}(\psi_{\chi}, \pi|_{K_0(\pp^m)})$ is trivial whenever $\chi \neq \chi_{\pi}$, this in turn is isomorphic to
\begin{align*}
\bigoplus_{\substack{\chi \in \widehat{\OO^{\times}} \\ 0 \leq c(\chi) \leq m}} \Hom_{K_0(\pp^m)}\left(\psi_{\chi}, \pi|_{K_0(\pp^m)}\right) & \cong \Hom_{K_0(\pp^m)}\left(\Ind_{K_1(\pp^m)}^{K_0(\pp^m)} 1,\pi|_{K_0(\pp^m)}\right)	\\
& \cong \Hom_{K_1(\pp^m)}\left(1, \pi|_{K_1(\pp^m)}\right)
\end{align*}
by \eqref{eqn:K0toK1ind} and Frobenius reciprocity. By \cite[Theorem 1]{Ree91}, we have that
\[\dim \Hom_{K_1(\pp^m)}\left(1, \pi|_{K_1(\pp^m)}\right) = \begin{dcases*}
\binom{m - c(\pi) + n - 1}{n - 1} & if $m \geq c(\pi)$,	\\
0 & otherwise.
\end{dcases*}\]
It follows that
\[\sum_{\ell = c(\chi_{\pi})}^{m} \dim \Hom_{K}\left(\tau_{\chi_{\pi},\ell}, \pi|_{K}\right) = \begin{dcases*}
\binom{m - c(\pi) + n - 1}{n - 1} & if $m \geq c(\pi)$,	\\
0 & otherwise,
\end{dcases*}\]
from which the result via induction in conjunction with the identity \eqref{eqn:binom}.
\end{proof}

\subsection{Matrix Coefficients}

We now study the properties of certain matrix coefficients of $\pi$ via matrix coefficients of $\HH_{\chi_{\pi},c(\pi)}(S^{n - 1})$. In order to do so, we explicitly identity $V_{\pi}^{\tau_{\chi_{\pi},c(\pi)}}$ with $\HH_{\chi_{\pi},c(\pi)}(S^{n - 1})$. We first observe an identity for the newform.

\begin{lemma}
The newform $v^{\circ} \in V_{\pi}^{\tau_{\chi_{\pi},c(\pi)}}$ of an induced representation of Langlands type $(\pi,V_{\pi})$ satisfies
\begin{equation}
\label{eqn:vcirc}
v^{\circ} = \dim \tau_{\chi_{\pi},c(\pi)} \int_{K} P_{\chi_{\pi},c(\pi)}^{\circ}(e_n k^{-1}) \pi(k) \cdot v^{\circ} \, dk.
\end{equation}
\end{lemma}

\begin{proof}
This is an immediate consequence of \hyperref[lem:xitauKn-11]{Lemma \ref*{lem:xitauKn-11}} and \hyperref[thm:newform]{Theorem \ref*{thm:newform}}.
\end{proof}

The finite-dimensional vector space $V_{\pi}^{\tau_{\chi_{\pi},c(\pi)}}$ admits a $K$-invariant inner product $\langle \cdot, \cdot \rangle$. We use such an inner product to construct an isomorphism between $\HH_{\chi_{\pi},c(\pi)}(S^{n - 1})$ and $V_{\pi}^{\tau_{\chi_{\pi},c(\pi)}}$.

\begin{proposition}
Let $(\pi,V_{\pi})$ be an induced representation of Langlands type with newform $v^{\circ} \in V_{\pi}^{\tau_{\chi_{\pi},c(\pi)}}$. An explicit isomorphism $P \mapsto v$ between $\HH_{\chi_{\pi},c(\pi)}(S^{n - 1})$ and $V_{\pi}^{\tau_{\chi_{\pi},c(\pi)}}$ is given by
\begin{equation}
\label{eqn:vdef}
v = \dim \tau_{\chi_{\pi},c(\pi)} \int_{K} P(e_n k^{-1}) \pi(k) \cdot v^{\circ} \, dk.
\end{equation}
For nonzero $v \in V_{\pi}^{\tau_{\chi_{\pi},c(\pi)}}$, the inverse is given by
\begin{equation}
\label{eqn:Pdef}
P(x) = \left(\dim \tau_{\chi_{\pi},c(\pi)}\right)^2 \langle P, P\rangle \int_{K} \frac{\left\langle \pi(k^{-1}) \cdot v, v^{\circ} \right\rangle}{\langle v, v\rangle} \left(\tau_{\chi_{\pi},c(\pi)}(k) \cdot P_{\chi_{\pi},c(\pi)}^{\circ}\right)(x) \, dk.
\end{equation}
Moreover, for all $P \in \HH_{\chi_{\pi},c(\pi)}(S^{n - 1})$ and $v \in V_{\pi}^{\tau_{\chi_{\pi},c(\pi)}}$ that are associated via \eqref{eqn:vdef} and \eqref{eqn:Pdef}, we have that
\begin{equation}
\label{eqn:matrixcoeff}
\langle P, P\rangle \left\langle \pi(k) \cdot v^{\circ}, v\right\rangle = \frac{\langle v, v\rangle \overline{P}(e_n k^{-1})}{\dim \tau_{\chi_{\pi},c(\pi)}}.
\end{equation}
\end{proposition}

\begin{proof}
We first confirm that for $v$ as in \eqref{eqn:vdef}, we have that $\Pi^{\tau_{\chi_{\pi},c(\pi)}}(v) = v$, so that $v \in V_{\pi}^{\tau_{\chi_{\pi},c(\pi)}}$. Indeed, from \eqref{eqn:vdef},
\begin{align*}
\Pi^{\tau_{\chi_{\pi},c(\pi)}}(v) & = \int_{K} \xi^{\tau_{\chi_{\pi},c(\pi)}}(k_1) \pi(k_1) \cdot v \, dk_1	\\
& = \dim \tau_{\chi_{\pi},c(\pi)} \int_{K} \xi^{\tau_{\chi_{\pi},c(\pi)}}(k_1) \int_{K} P(e_n k_2^{-1}) \pi(k_1 k_2) \cdot v^{\circ} \, dk_2 \, dk_1	\\
& = \dim \tau_{\chi_{\pi},c(\pi)} \int_{K} \pi(k_2) \cdot v^{\circ} \int_{K} \xi^{\tau_{\chi_{\pi},c(\pi)}}(k_1) \left(\tau_{\chi_{\pi},c(\pi)}(k_1) \cdot P\right)(e_n k_2^{-1}) \, dk_1 \, dk_2	\\
& = \dim \tau_{\chi_{\pi},c(\pi)} \int_{K} P(e_n k^{-1}) \pi(k) \cdot v^{\circ} \, dk	\\
& = v.
\end{align*}

Next, we show that \eqref{eqn:Pdef} follows from \eqref{eqn:matrixcoeff}. By \eqref{eqn:Qexpansion}, \eqref{eqn:reproducing}, \eqref{eqn:PchimcircL2}, and Schur orthogonality, we have that for any orthonormal basis $\{Q_j\}$ of $\HH_{\chi_{\pi},c(\pi)}(S^{n - 1})$,
\begin{align*}
P(x) & = \sum_{j = 1}^{\dim \HH_{\chi_{\pi},c(\pi)}(S^{n - 1})} \left\langle P, Q_j\right\rangle Q_j(x)	\\
& = \sum_{j = 1}^{\dim \HH_{\chi_{\pi},c(\pi)}(S^{n - 1})} Q_j(x) \left(\dim \tau_{\chi_{\pi},c(\pi)}\right)^2	\\
& \hspace{2cm} \times \int_{K} \left\langle \tau_{\chi_{\pi},c(\pi)}(k) \cdot P_{\chi_{\pi},c(\pi)}^{\circ}, Q_j\right\rangle \left\langle \tau_{\chi_{\pi},c(\pi)}(k^{-1}) \cdot P, P_{\chi_{\pi},c(\pi)}^{\circ}\right\rangle \, dk	\\
& = \dim \tau_{\chi_{\pi},c(\pi)} \int_{K} P_{\chi_{\pi},c(\pi)}^{\circ}(xk) P(e_n k^{-1}) \, dk.
\end{align*}
Assuming \eqref{eqn:matrixcoeff}, this is precisely \eqref{eqn:Pdef}.

It remains to prove \eqref{eqn:matrixcoeff}. We first note that
\begin{equation}
\label{eqn:vdef2}
v = \left(\dim \tau_{\chi_{\pi},c(\pi)}\right)^2 \int_{K} \left\langle \tau_{\chi_{\pi},c(\pi)}(k^{-1}) \cdot P, P_{\chi_{\pi},c(\pi)}^{\circ} \right\rangle \pi(k) \cdot v^{\circ} \, dk
\end{equation}
via \eqref{eqn:reproducing} and \eqref{eqn:vdef}. By \eqref{eqn:reproducing}, \eqref{eqn:vdef}, \eqref{eqn:vdef2}, and Schur orthogonality, we therefore have that
\begin{align*}
\langle P, P\rangle v & = \left(\dim \tau_{\chi_{\pi},c(\pi)}\right)^2 \langle P, P\rangle \int_{K} \left\langle \tau_{\chi_{\pi},c(\pi)}(k^{-1}) \cdot P, P_{\chi_{\pi},c(\pi)}^{\circ}\right\rangle \pi(k) \cdot v^{\circ} \, dk	\\
& = \left(\dim \tau_{\chi_{\pi},c(\pi)}\right)^3 \int_{K} \pi(k_1) \cdot v^{\circ}	\\
& \qquad \times \int_{K} \left\langle \tau_{\chi_{\pi},c(\pi)}(k_2) \cdot \left(\tau_{\chi_{\pi},c(\pi)}(k_1^{-1}) \cdot P\right), P\right\rangle \left\langle \tau_{\chi_{\pi},c(\pi)}(k_2^{-1}) \cdot P, P_{\chi_{\pi},c(\pi)}^{\circ}\right\rangle \, dk_2 \, dk_1	\\
& = \left(\dim \tau_{\chi_{\pi},c(\pi)}\right)^3 \int_{K} \left\langle \tau_{\chi_{\pi},c(\pi)}(k_1^{-1}) \cdot P, P\right\rangle	\\
& \qquad \times \int_{K} \left\langle \tau_{\chi_{\pi},c(\pi)}(k_2^{-1}) \cdot P, P_{\chi_{\pi},c(\pi)}^{\circ}\right\rangle \pi(k_1 k_2) \cdot v^{\circ} \, dk_2 \, dk_1	\\
& = \dim \tau_{\chi_{\pi},c(\pi)} \int_{K} \left\langle \tau_{\chi_{\pi},c(\pi)}(k^{-1}) \cdot P, P\right\rangle \pi(k) \cdot v \, dk.
\end{align*}
We also have from \eqref{eqn:vdef} that
\begin{equation}
\label{eqn:vv}
\langle v, v\rangle = \dim \tau_{\chi_{\pi},c(\pi)} \int_{K} \left\langle \pi(k) \cdot v^{\circ}, v\right\rangle P(e_n k^{-1}) \, dk.
\end{equation}
We therefore have by \eqref{eqn:reproducing}, \eqref{eqn:vcirc}, \eqref{eqn:vdef2}, \eqref{eqn:vv}, and Schur orthogonality that
\begin{align*}
\langle P, P\rangle \left\langle \pi(k) \cdot v^{\circ}, v\right\rangle & = \left(\dim \tau_{\chi_{\pi},c(\pi)}\right)^3 \int_{K} \int_{K} \left\langle \tau_{\chi_{\pi},c(\pi)}(k_1^{-1}) \cdot P_{\chi_{\pi},c(\pi)}^{\circ}, P_{\chi_{\pi},c(\pi)}^{\circ}\right\rangle\left\langle \tau_{\chi_{\pi},c(\pi)}(k_2) \cdot P, P\right\rangle	\\
& \hspace{4cm} \times \left\langle \pi(k k_1) \cdot v^{\circ}, \pi(k_2) \cdot v\right\rangle \, dk_1 \, dk_2	\\
& = \left(\dim \tau_{\chi_{\pi},c(\pi)}\right)^3 \int_{K} \left\langle \pi(k_1) \cdot v^{\circ}, v\right\rangle \int_{K} \left\langle \tau_{\chi_{\pi},c(\pi)}(k_2) \cdot P, \tau_{\chi_{\pi},c(\pi)}(k^{-1}) \cdot P\right\rangle	\\
& \hspace{1cm} \times \left\langle \tau_{\chi_{\pi},c(\pi)}(k_2^{-1}) \cdot P_{\chi_{\pi},c(\pi)}^{\circ}, \tau_{\chi_{\pi},c(\pi)}(k_1) \cdot P_{\chi_{\pi},c(\pi)}^{\circ}\right\rangle \, dk_2 \, dk_1	\\
& = \left(\dim \tau_{\chi_{\pi},c(\pi)}\right)^2 \left\langle P_{\chi_{\pi},c(\pi)}^{\circ}, \tau_{\chi_{\pi},c(\pi)}(k^{-1}) \cdot P\right\rangle	\\
& \hspace{3cm} \times \int_{K} \left\langle \pi(k_1) \cdot v^{\circ}, v\right\rangle \left\langle \tau_{\chi_{\pi},c(\pi)}(k_1^{-1}) \cdot P, P_{\chi_{\pi},c(\pi)}^{\circ}\right\rangle \, dk_1	\\
& = \frac{\langle v, v\rangle \overline{P}(e_n k^{-1})}{\dim \tau_{\chi_{\pi},c(\pi)}}.
\qedhere
\end{align*}
\end{proof}

An immediate consequence of the matrix coefficient identity \eqref{eqn:matrixcoeff} is the following.

\begin{corollary}
\label{cor:matrixcoeff}
Let $(\pi,V_{\pi})$ be an induced representation of Langlands type with newform $v^{\circ} \in V_{\pi}^{\tau_{\chi_{\pi},c(\pi)}}$. We have that
\begin{align*}
\frac{\left\langle \pi(k) \cdot v^{\circ}, v^{\circ}\right\rangle}{\langle v^{\circ}, v^{\circ}\rangle} & = P_{\chi_{\pi},c(\pi)}^{\circ}(e_n k)	\\
& = \begin{dcases*}
\chi_{\pi}(e_n k \prescript{t}{}{e}_n) & if $k \in K_0(\pp^{c(\pi)})$,	\\
\alpha_{\chi_{\pi},c(\pi);c(\pi) - 1} \chi_{\pi}(e_n k \prescript{t}{}{e}_n) & if $c(\pi) > c(\chi_{\pi})$ and $k \in K_0(\pp^{c(\pi) - 1}) \setminus K_0(\pp^{c(\pi)})$,	\\
0 & otherwise,
\end{dcases*}
\end{align*}
where $\alpha_{\chi_{\pi},c(\pi);c(\pi) - 1}$ is as in \eqref{eqn:alphachimm-1}.
\end{corollary}

\begin{proof}
We take $v = v^{\circ}$ in \eqref{eqn:matrixcoeff} and apply \eqref{eqn:Pchimcircdef}, \eqref{eqn:k-1bar}, and \eqref{eqn:PchimcircL2}.
\end{proof}

\begin{remark}
From \cite[Lemma 3.3]{Tem14}, we have that $c(\pi) = c(\chi_{\pi})$ if and only if $\pi$ is a twist-minimal principal series representation, so that $\pi = \Ind_{\Pgp(F)}^{\GL_n(F)} \bigboxtimes_{j = 1}^{n} \pi_j$ with each $\pi_j$ a character of $F^{\times}$ and $c(\pi_j) = 0$ for all but at most one $j \in \{1,\ldots,n\}$.
\end{remark}

\begin{remark}
Suppose that $\pi$ is a unitary generic irreducible admissible smooth representation of $\GL_n(F)$. One may take as a model for $\pi$ the Whittaker model $\WW(\pi,\psi)$; the $K$-invariant inner product on the $\tau_{\chi_{\pi},c(\pi)}$-isotypic subspace of $\WW(\pi,\psi)$ is then given by
\[\langle W_1,W_2\rangle \coloneqq \int\limits_{\Ngp_{n - 1}(F) \backslash \GL_{n - 1}(F)} W_1\begin{pmatrix} g & 0 \\ 0 & 1 \end{pmatrix} \overline{W_2}\begin{pmatrix} g & 0 \\ 0 & 1 \end{pmatrix} \, dg.\]
(In fact, this is a $\GL_n(F)$-invariant inner product \cite{Ber84}.) The canonical normalisation of the newform $W^{\circ}$ in the Whittaker model is such that $W^{\circ}(1_n) = 1$. Venkatesh \cite[Section 7]{Ven06} has explicitly calculated $\langle W^{\circ}, W^{\circ}\rangle$ with respect to this inner product; when $\pi$ is spherical, so that $c(\pi) = 0$, this is simply $L(1,\ad \pi)$. Via \hyperref[cor:matrixcoeff]{Corollary \ref*{cor:matrixcoeff}}, this allows one to explicitly determine $\langle \pi(k) \cdot W^{\circ}, W^{\circ}\rangle$.
\end{remark}

\section{Archimedean Analogues}
\label{sect:archimedean}

\subsection{Archimedean Spherical Harmonics}

The archimedean analogues of the results in \hyperref[sect:spherical]{Sections \ref*{sect:spherical}} and \ref{sect:zonal} are well-known; we briefly survey them for the sake of comparison. In place of a nonarchimedean field $F$, we instead work with an archimedean field, which is either $\R$ or $\C$.

\subsubsection{$F = \R$}

The maximal compact subgroup $K_n$ of $\GL_n(\R)$ is the orthogonal group $\Ogp(n)$. This acts transitively on unit sphere
\[S^{n - 1} \coloneqq \left\{x = (x_1,\ldots,x_n) \in \R^n : x_1^2 + \cdots + x_n^2 = 1\right\}\]
in $\R^n$ via the group action $k \cdot x \coloneqq xk$ for $k \in K_n$ and $x \in S^{n - 1}$. The stabiliser subgroup of $K_n$ with respect to the point $e_n \coloneqq (0,\ldots,0,1) \in S^{n - 1}$ is
\begin{equation}
\label{eqn:archKn-1}
\left\{ \begin{pmatrix} k' & 0 \\ 0 & 1 \end{pmatrix} \in K_n : k' \in K_{n - 1}\right\},
\end{equation}
which we freely identify with $K_{n - 1}$. It follows that $S^{n - 1} \cong K_{n - 1} \backslash K_n$.

Note that the subgroup \eqref{eqn:archKn-1} does not, at first glance, appear to be the direct archimedean analogue of the subgroup $K_{n - 1,1}$ as in \eqref{eqn:Kn-11}, since $K_{n - 1,1} \cong K_{n - 1} \ltimes \OO^{n - 1}$. In the archimedean setting, on the other hand, there is no analogue of $\OO^{n - 1}$; nonetheless, the subgroup \eqref{eqn:archKn-1}, like the nonarchimedean subgroup \eqref{eqn:Kn-11}, is the maximal compact subgroup of the mirabolic subgroup \eqref{eqn:mirabolic}.

The decomposition of the right regular representation of $K_n$ on $C^{\infty}(S^{n - 1})$ is precisely the theory of spherical harmonics. This is best understood in terms of homogeneous harmonic polynomials, for which we follow \cite[Chapter 2]{AH12}.

Given a nonnegative integer $m$, let $\HH_m(\R^n)$ denote the vector space consisting of homogeneous harmonic polynomials of degree $m$, namely the set of polynomials $P(x_1,\ldots,x_n)$ in $(x_1,\ldots,x_n) \in \R^n$ that are annihilated by the Laplacian
\[\Delta = \sum_{j = 1}^{n} \frac{\dee^2}{\dee x_j^2}\]
and satisfy $P(\lambda x_1,\ldots,\lambda x_n) = \lambda^m P(x_1,\ldots,x_n)$ for all $\lambda \in \R$. This space has dimension $1$ for $n = 1$ and $m \in \{0,1\}$ and has dimension
\[\binom{m + n - 2}{n - 2} + \binom{m + n - 3}{n - 3} = \frac{(2m + n - 2) (m + n - 3)!}{m!(n - 2)!} = \frac{2m + n - 2}{m + n - 2} \binom{m + n - 2}{n - 2}\]
for $n \geq 2$.

Let $\HH_m(S^{n - 1})$ denote the vector space of the restriction of elements of $\HH_m(\R^n)$ to the unit sphere; via the homogeneity of elements of $\HH_m(\R^n)$, these spaces are isomorphic. Elements of $\HH_m(S^{n - 1})$ are called spherical harmonics of degree $m$. The group $K_n$ acts on $\HH_m(\R^n)$ via right translation, which descends to an action on $\HH_m(S^{n - 1})$. As a $K_n$-module, $\HH_m(S^{n - 1})$ is irreducible, and for $n \geq 2$, we have the orthogonal decomposition
\[C^{\infty}(S^{n - 1}) = \bigoplus_{m = 0}^{\infty} \HH_m(S^{n - 1}),\]
analogous to the decomposition \eqref{eqn:Cdecomp} of $C^{\infty}(S^{n - 1})$.

We may view $\HH_m(S^{n - 1})$ as the real analogue of $\HH_{\chi,m}(S^{n - 1})$. In turn, the real analogue of $C^{\infty}(S^{n - 1})^{K(\pp^m)}$ is $\bigoplus_{\ell = 0}^{m} \HH_{\ell}(S^{n - 1})$, the space of spherical harmonics of degree at most $m$. As $K_1 = \Ogp(1) \cong \Z/2\Z$, which is the real analogue of $\OO^{\times}$, characters of $K_1$ are of the form $\sgn^{\kappa}$ for $\kappa \in \{0,1\}$, where $\sgn(x) \coloneqq x/|x|$. In particular, the central character $\chi$ of the $K_n$-module $\HH_m(S^{n - 1})$ is simply $\sgn^{m \pmod{2}}$, which is determined by the parity of the nonnegative integer $m$. So the real analogue of $C^{\infty}(S^{n - 1})_{\chi}^{K(\pp^m)}$, in terms of its orthogonal decomposition \eqref{eqn:CchiKmdecomp}, is precisely
\[\bigoplus_{\substack{\ell = 0 \\ \ell \equiv \kappa \hspace{-.25cm} \pmod{2}}}^{m} \HH_{\ell}(S^{n - 1})\]
for $\chi = \sgn^{\kappa}$ with $\kappa \in \{0,1\}$ and $m \geq \kappa$, while the analogue the orthogonal decomposition \eqref{eqn:CKmdecomp} of $C^{\infty}(S^{n - 1})^{K(\pp^m)}$ is simply
\[\bigoplus_{\ell = 0}^{m} \HH_{\ell}(S^{n - 1}) = \bigoplus_{\kappa \in \{0,1\}} \bigoplus_{\substack{\ell = 0 \\ \ell \equiv \kappa \hspace{-.25cm} \pmod{2}}}^{m} \HH_{\ell}(S^{n - 1}).\]

There exists a unique spherical harmonic $P_m^{\circ} \in \HH_m(S^{n - 1})$ satisfying $P_m^{\circ}(e_n) = 1$ and $P_m^{\circ}(xk) = P_m^{\circ}(x)$ for all $x = (x_1,\ldots,x_n) \in S^{n - 1}$ and $k = \begin{psmallmatrix} k' & 0 \\ 0 & 1 \end{psmallmatrix} \in K_n$ with $k' \in K_{n - 1}$, namely
\[P_m^{\circ}(x_1,\ldots,x_n) \coloneqq \sum_{\substack{\nu = 0 \\ \nu \equiv 0 \hspace{-.25cm} \pmod{2}}}^{m} \frac{i^{\nu} m! \Gamma\left(\frac{n - 1}{2}\right)}{2^{\nu} \left(\frac{\nu}{2}\right)! (m - \nu)! \Gamma\left(\frac{\nu + n - 1}{2}\right)} \left(x_1^2 + \cdots + x_{n - 1}^2\right)^{\frac{\nu}{2}} x_n^{m - \nu}.\]
This is the zonal spherical harmonic on $S^{n - 1}$ of degree $m$, which is the real analogue of the zonal spherical function $P_{\chi,m}^{\circ} \in \HH_{\chi,m}(S^{n - 1})$. From this, the addition theorem for $\HH_m(S^{n - 1})$ takes precisely the same form as the nonarchimedean result given in \hyperref[lem:addthm]{Lemma \ref*{lem:addthm}}; similarly, $(\dim \HH_m(S^{n - 1})) P_m^{\circ}$ is the reproducing kernel for $\HH_m(S^{n - 1})$, akin to \hyperref[cor:reproducing]{Corollary \ref*{cor:reproducing}}.

\subsubsection{$F = \C$}

The maximal compact subgroup $K_n$ of $\GL_n(\C)$ is the unitary group $\Ugp(n)$. This acts transitively on unit sphere
\[S^{n - 1} \coloneqq \left\{z = (z_1,\ldots,z_n) \in \C^n : z_1 \overline{z_1} + \cdots + z_n \overline{z_n} = 1\right\}\]
in $\C^n$ via the group action $k \cdot x \coloneqq xk$ for $k \in K_n$ and $x \in S^{n - 1}$. (It behoves us to point out that as a real topological manifold, this should be viewed as the $(2n - 1)$-dimensional unit sphere, but we use the notation $S^{n - 1}$ for the sake of consistency.) Just as for the real case, the stabiliser subgroup of $K_n$ with respect to the point $e_n \coloneqq (0,\ldots,0,1) \in S^{n - 1}$ is the subgroup \eqref{eqn:archKn-1}, which we freely identify with $K_{n - 1}$, so that $S^{n - 1} \cong K_{n - 1} \backslash K_n$.

The decomposition of the right regular representation of $K_n$ on $C^{\infty}(S^{n - 1})$ is again the theory of spherical harmonics, with the additional complexification that one must consider the \emph{bidegree} of such a spherical harmonic. This is best understood in terms of homogeneous harmonic polynomials, for which we follow \cite[Chapter 12]{Rud08}.

Given a pair of nonnegative integers $m_1,m_2$, let $\HH_{m_1,m_2}(\C^n)$ denote the vector space consisting of homogeneous harmonic polynomials of bidegree $(m_1,m_2)$, namely the set of polynomials $P(z_1,\ldots,z_n,\overline{z_1},\ldots,\overline{z_n})$ in $(z_1,\ldots,z_n) \in \C^n$ that are annihilated by the Laplacian
\[\Delta = 4\sum_{j = 1}^{n} \frac{\dee^2}{\dee z_j \dee \overline{z_j}}\]
and satisfy $P(\lambda z_1,\ldots,\lambda z_n, \overline{\lambda z_1}, \ldots, \overline{\lambda z_n}) = \lambda^{m_1} \overline{\lambda}^{m_2} P(z_1,\ldots,z_n,\overline{z_1},\ldots,\overline{z_n})$ for all $\lambda \in \C$. This space has dimension $1$ for $n = 1$ and dimension
\[\frac{(m_1 + m_2 + n - 1) (m_1 + n - 2)! (m_2 + n - 2)!}{m_1! m_2! (n - 2)! (n - 1)!} = \frac{m_1 + m_2 + n - 1}{n - 1} \binom{m_1 + n - 2}{n - 2} \binom{m_2 + n - 2}{n - 2}\]
for $n \geq 2$.

Let $\HH_{m_1,m_2}(S^{n - 1})$ denote the vector space of the restriction of elements of $\HH_{m_1,m_2}(\C^n)$ to the unit sphere; via the homogeneity of elements of $\HH_{m_1,m_2}(\C^n)$, these spaces are isomorphic. Elements of $\HH_{m_1,m_2}(S^{n - 1})$ are called spherical harmonics of bidegree $(m_1,m_2)$. The group $K_n$ acts on $\HH_{m_1,m_2}(\C^n)$ via right translation, which descends to an action on $\HH_{m_1,m_2}(S^{n - 1})$. As a $K_n$-module, $\HH_{m_1,m_2}(S^{n - 1})$ is irreducible, and for $n \geq 2$, we have the orthogonal decomposition
\[C^{\infty}(S^{n - 1}) = \bigoplus_{m_1,m_2 = 0}^{\infty} \HH_{m_1,m_2}(S^{n - 1}),\]
analogous to the decomposition \eqref{eqn:Cdecomp} of $C^{\infty}(S^{n - 1})$.

We may view $\HH_{m_1,m_2}(S^{n - 1})$ as the complex analogue of $\HH_{\chi,m}(S^{n - 1})$. In turn, the complex analogue of $C^{\infty}(S^{n - 1})^{K(\pp^m)}$ is $\bigoplus_{m_1 + m_2 \leq m} \HH_{m_1,m_2}(S^{n - 1})$, the space of spherical harmonics of total degree at most $m$. As $K_1 = \Ugp(1) \cong \R/\Z$, which is the complex analogue of $\OO^{\times}$, characters of $K_1$ are of the form $e^{i\kappa\arg}$ for $\kappa \in \Z$, where $e^{i\arg(z)} \coloneqq z^{1/2} \overline{z}^{-1/2}$. In particular, the central character $\chi$ of the $K_n$-module $\HH_{m_1,m_2}(S^{n - 1})$ is $e^{i(m_1 - m_2)\arg}$, which is determined by the difference of $m_1$ and $m_2$. So the complex analogue of $C^{\infty}(S^{n - 1})_{\chi}^{K(\pp^m)}$, in terms of its orthogonal decomposition \eqref{eqn:CchiKmdecomp}, is precisely
\[\bigoplus_{\substack{m_1 + m_2 \leq m \\ m_1 - m_2 = \ell}} \HH_{m_1,m_2}(S^{n - 1})\]
for $\chi = e^{i\ell\arg}$ with $m \geq |\ell|$, while the orthogonal decomposition \eqref{eqn:CKmdecomp} of $C^{\infty}(S^{n - 1})^{K(\pp^m)}$ has the complex analogue
\[\bigoplus_{m_1 + m_2 \leq m} \HH_{m_1,m_2}(S^{n - 1}) = \bigoplus_{|\ell| \leq m} \bigoplus_{\substack{m_1 + m_2 \leq m \\ m_1 - m_2 = \ell}} \HH_{m_1,m_2}(S^{n - 1}).\]

There exists a unique spherical harmonic $P_m^{\circ} \in \HH_{m_1,m_2}(S^{n - 1})$ satisfying $P_{m_1,m_2}^{\circ}(e_n,e_n) = 1$ and $P_{m_1,m_2}^{\circ}(zk,\overline{zk}) = P_{m_1,m_2}^{\circ}(z,\overline{z})$ for all $z = (z_1,\ldots,z_n) \in S^{n - 1}$ and $k = \begin{psmallmatrix} k' & 0 \\ 0 & 1 \end{psmallmatrix} \in K_n$ with $k' \in K_{n - 1}$, namely
\[P_{m_1,m_2}^{\circ}(z,\overline{z}) \coloneqq \sum_{\nu = 0}^{\min\{m_1,m_2\}} \frac{(-1)^{\nu} \binom{m_1}{\nu} \binom{m_2}{\nu}}{\binom{\nu + n - 2}{n - 2}} \left(z_1 \overline{z_1} + \cdots + z_{n - 1} \overline{z_{n - 1}}\right)^{\nu} z_n^{m_1 - \nu} \overline{z_n}^{m_2 - \nu}.\]
This is the zonal spherical harmonic on $S^{n - 1}$ of bidegree $(m_1,m_2)$, which is the complex analogue of the zonal spherical function $P_{\chi,m}^{\circ} \in \HH_{\chi,m}(S^{n - 1})$. Once more, the addition theorem for $\HH_{m_1,m_2}(S^{n - 1})$ takes the same form as \hyperref[lem:addthm]{Lemma \ref*{lem:addthm}} and $(\dim \HH_m(S^{n - 1})) P_m^{\circ}$ is the reproducing kernel for $\HH_m(S^{n - 1})$.

\subsection{Archimedean Newform Theory for \texorpdfstring{$\GL_n$}{GL\9040\231}}

Finally, we consider the archimedean analogues of the results in \hyperref[sect:Ktype]{Section \ref*{sect:Ktype}}; these are due to Popa \cite{Pop08} for $\GL_2$ and to the author \cite{Hum20} for $\GL_n$ with $n$ arbitrary.

Let $(\pi,V_{\pi})$ be an induced representation of Langlands type of $\GL_n(F)$, where $F$ is an archimedean local field, so that $F$ is either $\R$ or $\C$. There is no obvious analogue in the archimedean setting of the nonarchimedean congruence subgroups $K_1(\pp^m)$ and $K_0(\pp^m)$ of $K_n$ (though cf.~\cite{JN19}). This prevents one from defining the newform and conductor exponent of $\pi$ via the subspace $V_{\pi}^{K_1(\pp^{c(\pi)})}$ of $K_1(\pp^{c(\pi)})$-invariant vectors as in \hyperref[def:newform]{Definition \ref*{def:newform}}.

As highlighted in \hyperref[rem:othercharacterisation]{Remark \ref*{rem:othercharacterisation}}, one can instead characterise the nonarchimedean newform, when viewed in the Whittaker model, as the unique Whittaker function $W^{\circ} \in \WW(\pi,\psi)$ that is both right $K_{n - 1}$-invariant, with $K_{n - 1}$ embedded in $K_n$ as the subgroup \eqref{eqn:archKn-1}, and is a test vector for the local $\GL_n \times \GL_{n - 1}$ Rankin--Selberg integral whenever the second representation is unramified. It is shown in \cite[Theorem 4.17]{Hum20} that this characterisation does indeed also hold in the archimedean setting, in that there is a unique such Whittaker function satisfying these two conditions. One would also like to show that the conductor exponent is characterised via the epsilon factor $\e(s,\pi,\psi)$; while this is the case when $F$ is nonarchimedean, this is unfortunately insufficient when $F$ is archimedean, for then $\e(s,\pi,\psi)$ is simply an integral power of $i$, and this integer is only determined modulo $4$.

Nonetheless, the archimedean analogue of \hyperref[thm:newform]{Theorem \ref*{thm:newform}} holds; more precisely, the archimedean analogues of \eqref{eqn:dimVKn-11} and \eqref{eqn:VKn-11} are true via \cite[Theorems 4.7 and 4.12]{Hum20}. Here the nonarchimedean subgroup $K_{n - 1,1}$ of $K_n$ is replaced by the archimedean subgroup \eqref{eqn:archKn-1}, which both share the property that they are the stabiliser subgroup of $K_n$ with respect to the point $e_n \in F^n$; moreover, the notion of ordering $K_n$-types by their \emph{level} in the nonarchimedean setting is replaced by ordering $K_n$-types by their \emph{Howe degree} in the archimedean setting (cf.~\cite[Section 4.2]{Hum20}).

In \cite[Definition 4.8]{Hum20}, the author has \emph{defined} the conductor exponent $c(\pi)$ of an induced representation of Langlands type $(\pi,V_{\pi})$ of $\GL_n$ over an archimedean field $F$ to be the minimal nonnegative integer $m$ for which $\pi$ contains a $K_n$-type $\tau$ of Howe degree $m$ having a nontrivial $K_{n - 1}$-fixed vector, with $K_{n - 1}$ embedded in $K_n$ as the subgroup \eqref{eqn:archKn-1}. This distinguished $K_n$-type $\tau = \tau^{\circ}$ is the newform $K_n$-type and appears with multiplicity one in $\pi$. The author has also \emph{defined} the newform in the archimedean setting to be the nonzero vector in $V_{\pi}$, unique up to scalar multiplication, that is invariant under the subgroup \eqref{eqn:archKn-1} and is $\tau^{\circ}$-isotypic.

We observed in \hyperref[rem:newformdef]{Remark \ref*{rem:newformdef}} that \hyperref[thm:newform]{Theorem \ref*{thm:newform}} gives alternative characterisations of the newform and conductor exponent in the nonarchimedean setting in terms of $K_n$-types. In conjunction with \cite[Theorem 4.7]{Hum20}, this characterisation thereby unifies the nonarchimedean and archimedean treatments of the newform and the conductor exponent.

\phantomsection
\addcontentsline{toc}{section}{Acknowledgements}
\hypersetup{bookmarksdepth=-1}

\subsection*{Acknowledgements}

Thanks are owed to Subhajit Jana for helpful discussions regarding newforms and matrix coefficients.

\hypersetup{bookmarksdepth}

\end{document}